\documentclass[11pt]{amsart}
\usepackage[T1]{fontenc}
\usepackage{lmodern}
\usepackage{amsmath,amssymb,amsthm,mathtools,mathrsfs}
\usepackage{microtype}
\usepackage[hidelinks]{hyperref}
\usepackage[top=.93in,bottom=.93in,left=1.03in,right=1.03in]{geometry}
\hypersetup{pdftitle={Geometric Duality Between Constraints and Gauge Fields: Mirror Realization and Reduction Geometry on Principal Bundles},pdfauthor={Dongzhe Zheng}}

\newcommand{\g}{\mathfrak g}
\newcommand{\h}{\mathfrak h}
\newcommand{\Ad}{\operatorname{Ad}}
\newcommand{\ad}{\operatorname{ad}}
\newcommand{\Gau}{\operatorname{Gau}}
\newcommand{\Hol}{\operatorname{Hol}}
\newcommand{\Char}{\operatorname{Char}}
\newcommand{\rank}{\operatorname{rank}}
\newcommand{\ord}{\operatorname{ord}}
\newcommand{\Fr}{\operatorname{Fr}}
\newcommand{\SO}{\operatorname{SO}}
\newcommand{\SU}{\operatorname{SU}}
\newcommand{\U}{\operatorname{U}}
\newcommand{\vol}{\operatorname{vol}}
\newcommand{\dd}{\mathrm d}
\newcommand{\R}{\mathbb R}
\newcommand{\Z}{\mathbb Z}
\newcommand{\C}{\mathbb C}
\newcommand{\id}{\operatorname{id}}

\theoremstyle{plain}
\newtheorem{theorem}{Theorem}[section]
\newtheorem{proposition}[theorem]{Proposition}
\newtheorem{corollary}[theorem]{Corollary}
\newtheorem{lemma}[theorem]{Lemma}
\theoremstyle{definition}

\newtheorem{example}[theorem]{Example}
\theoremstyle{remark}
\newtheorem{remark}[theorem]{Remark}

\title[Mirror realization and reduction geometry]{Geometric Duality Between Constraints and Gauge Fields:\\Mirror Realization and Reduction Geometry on Principal Bundles}
\author{Dongzhe Zheng}
\address{Princeton University}
\email{dz5992@princeton.edu}
\date{September 2026}

\begin{document}
\raggedbottom
\begin{abstract}
A connection and a nonzero parallel adjoint field determine an invariant hyperplane constraint on a principal bundle.  Its sign mirror preserves the hyperplane and reverses its coorientation; global gauge realization is controlled by a twisted stabilizer reduction.  For regular fields we identify the normalizing gauge extension as a pushout of the torus-normalizer extension, giving exact lift orders and simultaneous-splitting criteria.  In singular rank-two block families, reductions on a fixed trivial bundle form an affine second-Chern lattice whose Weyl stabilizers and finite-order lift spectra detect topology invisible to paired curvature.  The reduction framework also determines the structure group and second cohomology of the matched-flag diagonalization space of Friedman and Park, and gives a first- and second-Chern criterion for normal matrices with fixed separated spectrum on four-complexes; every integral solution of their three-eigenline equation on $S^2\times S^2$ is realized.  For moving reductions, the projected circle curvature differs from the ambient paired curvature by a covariant-derivative term.  Full fatness on a closed four-manifold forces a nontrivial sign-mirror obstruction for every circle reduction; hyperbolic self-dual-form bundles also provide circle reductions in the $y$-fat setting of Florit and Ziller.  Contact transgression, bundle automorphism twists, and the natural first-jet Spencer operator complete the geometric picture.
\end{abstract}
\maketitle

\section{Introduction}

Let $P\to M$ be a principal bundle, $A$ a connection, and $u$ a nonzero adjoint field.  The scalar form $\alpha=B(u,A)$ specifies a cooriented constraint hyperplane $\mathcal D=\ker\alpha$.  We call $(P,A,u)$ a \emph{parallel compatible pair} when $\dd_Au=0$.  Its stabilizer reduction $Q=u^{-1}(\mu)$ is then preserved by $A$.  The sign operation $u\mapsto-u$ preserves $\mathcal D$ and reverses its coorientation.  We use \emph{mirror} for this sign operation and, when appropriate, for its Weyl or bundle-automorphism analogues.  A mirror need not be induced by a gauge transformation; the difference is measured by the topology of $Q$.

The reduction and its connection supply the common geometry throughout the paper.  In the parallel regime $\dd_Au=0$, normalizer elements act on the class of the $A$-preserved reduction and impose exact tests for global gauge lifts.  The bundle-level normalizer criterion itself does not require parallelism and later applies to moving reductions as well.  In the regular case the class lies in a torus Chern lattice; in the singular case it is a nonabelian bundle class, whose higher Chern data can survive even when the ambient bundle is trivial.  In the moving regime $\dd_Au\ne0$, the reduction remains a reduction of the underlying bundle, and its covariant derivative contributes to the curvature of the projected connection.  A fully fat $\SO(3)$ connection has no nonzero parallel adjoint field (Lemma~\ref{lem:no-parallel-fat}); its proper circle reductions necessarily belong to the moving regime.

The central structural statement is the normalizer pushout in Theorem~\ref{thm:pushout}.  Theorem~\ref{thm:diagonalization-base} and Proposition~\ref{prop:universal-obstruction} apply it to the matched-flag space introduced by Friedman and Park~\cite{FriedmanPark2016}.  Theorem~\ref{thm:four-dimensional-realization} provides the four-dimensional spectral realization criterion, including their $S^2\times S^2$ question~\cite{FriedmanPark2024}.  Theorem~\ref{thm:charge-lattice} exhibits the affine second-Chern charge lattice and its exact finite symmetry groups.  Theorem~\ref{thm:fat-mirror} links fatness to Weyl sign obstruction, and Theorem~\ref{thm:hyperbolic} gives a moving-reduction example for the question of Ziller~\cite{ZillerFatness} and its strengthening by Florit and Ziller~\cite{FZ}.  The proofs identify their hypotheses and the relevant bundle topology explicitly.

The contact case of Proposition~\ref{prop:levi} specializes Lerman's contact coupling criterion~\cite{Lerman}; its general Levi-rank formula is proved directly.  The pointwise normalizer problem has roots in Tits' extension and the lift-order work of Adams and He~\cite{Tits,AdamsHe}.  Blau and Thompson established the importance of torus sectors and global diagonalization obstructions~\cite{BlauThompson}; related normalizer reductions and invariant connections appear in~\cite{BiswasMachu}.  Fine and Panov developed the definite-connection framework used for the four-dimensional fatness construction~\cite{FP}; the almost-complex existence criterion comes from Kotschick~\cite{Kotschick}.  We retain the standard Atiyah-algebroid cochain model and first-jet Spencer naturality in the appendices to state precisely the cohomological functoriality available for these compatible pairs~\cite{Atiyah,Mackenzie,Spencer,CSS}.

Unless explicitly stated otherwise, bundles, fields, and maps are smooth, $M$ is connected and paracompact, and gauge transformations cover the identity of $M$.  The structure group in Sections~\ref{sec:geometry}--\ref{sec:contact} is compact, connected, and semisimple.  In the normalizer and matrix sections, $\U(n)$ is also allowed.  Four-dimensional spectral classification is stated for finite CW complexes, where the bundle and matrix maps are continuous.  We write $F^*u=u\circ F$ for the pullback action on an equivariant adjoint field.  Under an associated vector-bundle realization this is conjugation by the inverse of the bundle map induced by $F$; statements about existence, stabilizers, and orders do not depend on which of these equivalent actions is used.

\section{Compatible hyperplanes and their Levi form}\label{sec:geometry}

All manifolds and bundles are smooth.  Let $M$ be connected and let $\pi:P\to M$ be a right principal bundle for a compact connected semisimple Lie group $G$.  Set $\g=\operatorname{Lie}(G)$ and choose the positive definite inner product $B=-\kappa$, where $\kappa$ is the Killing form.  Every automorphism of $G$ preserves $B$.  We use standard connection and holonomy conventions~\cite{KN}.  For $a\in\g$, its fundamental vector field is
\[
 a_p^\#=\left.\frac{\dd}{\dd t}\right|_{t=0}p\exp(ta).
\]
A connection form $A\in\Omega^1(P;\g)$ obeys $A(a^\#)=a$ and $R_g^*A=\Ad_{g^{-1}}A$.  Write
\[
 F_A=\dd A+\tfrac12[A\wedge A],\qquad
 \dd_Au=\dd u+[A,u].
\]
An adjoint field is an equivariant map $u:P\to\g$, with $u(pg)=\Ad_{g^{-1}}u(p)$.  We assume throughout that
\begin{equation}\label{eq:parallel}
 \dd_Au=0,\qquad u\not\equiv0.
\end{equation}
Its norm is a positive constant $c=|u|_B$.  Define
\begin{equation}
 \alpha=B(u,A),\qquad \mathcal D=\ker\alpha.
\end{equation}
The form $\alpha$ is invariant and nowhere zero.  We call $(P,A,u)$ a \emph{parallel compatible pair}, and $\mathcal D$ its constraint hyperplane.  The coadjoint notation is $\lambda=B(u,\cdot)$; our convention $(\ad_a^*\lambda)(b)=-\lambda([a,b])$ turns~\eqref{eq:parallel} into $\dd\lambda+\ad_A^*\lambda=0$.

Let $\mathcal H=\ker A$ and $V=\ker\dd\pi$.  At $p\in P$, put $\mu=u(p)$ and $\g_\mu=\{a\in\g:[a,\mu]=0\}$.  Since $A|_V$ is an isomorphism,
\begin{equation}\label{eq:hyperplanesplit}
 \mathcal D_p=\mathcal H_p\oplus(\mu^\perp)^\#_p,
 \qquad \dim(\mathcal D_p\cap V_p)=\dim G-1.
\end{equation}
Thus $\mathcal D$ is a hyperplane; $\mathcal H$ is the horizontal splitting of the Atiyah sequence.

Parallelism and the Bianchi identity imply that $B(u,F_A)$ is horizontal, invariant, and closed.  There is consequently a closed $\beta\in\Omega^2(M)$ such that
\begin{equation}\label{eq:beta}
 \pi^*\beta=B(u,F_A).
\end{equation}

\begin{proposition}[Levi decomposition]\label{prop:levi}
For $v,w\in T_{\pi(p)}M$ and $a,b\in\mu^\perp$,
\begin{equation}\label{eq:levi}
 \dd\alpha_p(v^H+a^\#,w^H+b^\#)
 =\beta_{\pi(p)}(v,w)+B(\mu,[a,b]).
\end{equation}
Consequently
\[
 \Char(\mathcal D)_p=(\ker\beta_{\pi(p)})^H
 \oplus(\g_\mu\cap\mu^\perp)^\#_p,
 \quad
 \rank(\dd\alpha|_{\mathcal D_p})
 =\rank\beta_{\pi(p)}+\dim(G\cdot\mu).
\]
If $\dim M$ is even, $\alpha$ is contact precisely when $\beta$ is symplectic and $\dim\g_\mu=1$.
\end{proposition}

\begin{proof}
For $U,V\in T_pP$, put $a=A(U)$ and $b=A(V)$.  From $\dd u=-[A,u]$ and invariance of $B$,
\[
 \dd B(u,A)(U,V)=B(u,F_A(U,V))+B(u,[a,b]).
\]
The curvature term is horizontal.  The radical of $(a,b)\mapsto B(\mu,[a,b])$ on $\mu^\perp$ is $\g_\mu\cap\mu^\perp$: indeed $B(\mu,[a,b])=B([\mu,a],b)$, and $[\mu,a]$ is already orthogonal to $\mu$.  Its rank is $\dim\g-\dim\g_\mu=\dim(G\cdot\mu)$.  The two blocks give the formulas.  Nondegeneracy of $\dd\alpha|_{\mathcal D}$ is equivalent to the contact condition.
\end{proof}

\begin{proposition}[Constraint--parallel-field correspondence]\label{prop:projective}
Fix $A$ and $p_0\in P$, and let $K=\Hol_{p_0}(A)$.  Evaluation at $p_0$ identifies parallel adjoint fields with the fixed subspace $W=\g^K$.  Nonzero fields determine a family of invariant constraint hyperplanes, and
\[
 \mathbb P(W)\xrightarrow{\ \cong\ }
 \{\ker B(u,A):\dd_Au=0,\ u\ne0\}
\]
is a bijection.  Positive rays in $W$ classify these hyperplanes with their coorientation induced by $B(u,A)$.
\end{proposition}

\begin{proof}
Parallel transport extends any $K$-fixed value at $p_0$ uniquely to a parallel field, and equivariance extends it along each fiber.  If $\ker B(u,A)=\ker B(v,A)$, their restrictions to $V_{p_0}$ agree as hyperplanes.  Since $A_{p_0}:V_{p_0}\to\g$ is an isomorphism, $u(p_0)$ and $v(p_0)$ span the same line.  Uniqueness of parallel extension gives $v=t u$ for a nonzero real constant $t$.  The sign of $t$ records whether the coorientations agree.
\end{proof}

The sign operation
\begin{equation}
 (A,u)\longmapsto(A,-u)
\end{equation}
preserves parallelism and $\mathcal D$.  It sends $(\alpha,\beta)$ to $(-\alpha,-\beta)$.  Its behavior under gauge transformations is subtler because gauge transformations act on both the connection and the adjoint field.

\section{The sign-realization criterion}\label{sec:sign}

Fix $p_0\in P$, put $\mu=u(p_0)$ and $H=G_\mu$.  Parallel transport and equivariance show that
\begin{equation}
 Q=\{p\in P:u(p)=\mu\}
\end{equation}
is a principal $H$-reduction of $P$, and $A$ restricts to an $H$-connection on $Q$.  For $\vartheta\in\operatorname{Aut}(H)$, write $Q^\vartheta$ for the same total space with right action $q\star h=q\vartheta^{-1}(h)$.  Here and below gauge transformations cover $\id_M$.

\begin{theorem}[Global sign realization]\label{thm:signcriterion}
There exists a gauge transformation $F:P\to P$ satisfying
\begin{equation}\label{eq:signgauge}
 u(F(p))=-u(p)\qquad(p\in P)
\end{equation}
if and only if both of the following hold:
\begin{enumerate}
\item $-\mu\in G\cdot\mu$;
\item for one, equivalently every, $n\in G$ with $\Ad_{n^{-1}}\mu=-\mu$, the $H$-bundle $Q$ admits a $\vartheta_n$-semilinear automorphism over $\id_M$, where $\vartheta_n=\Ad_n|_H$.
\end{enumerate}
Equivalently, $Q$ is isomorphic to its twist $Q^{\vartheta_n}$.  In local sections $q_j=q_i h_{ij}$, condition~(2) says that there are smooth $t_i:U_i\to H$ satisfying
\begin{equation}\label{eq:cechsign}
 h_{ij}t_j=t_i\vartheta_n(h_{ij})
 \quad\text{on }U_i\cap U_j.
\end{equation}
\end{theorem}

\begin{proof}
The first condition follows from~\eqref{eq:signgauge} at $p_0$.  If $n$ implements the sign, it normalizes $H$, since $G_\mu=G_{-\mu}$.  For $q\in Q$ we have $F(q)\in Qn$.  Write $F(q)=f(q)n$ with $f(q)\in Q$.  Equivariance of $F$ gives
\[
 f(qh)n=F(qh)=F(q)h=f(q)nh,
 \qquad f(qh)=f(q)\vartheta_n(h).
\]
Thus $f$ is the required semilinear map.  Conversely such an $f$ extends to $P=Q\times_HG$ by
\begin{equation}\label{eq:extendF}
 F(qg)=f(q)ng.
\end{equation}
The semilinearity identity makes~\eqref{eq:extendF} independent of the chosen expression $qg$.  It is an invertible gauge transformation and satisfies~\eqref{eq:signgauge}.  In a local section write $f(q_i)=q_i t_i$ and compare the two expressions for $f(q_j)$ to obtain~\eqref{eq:cechsign}.  Replacing $n$ by another sign representative multiplies it by an element of $H$ and changes $\vartheta_n$ by an inner automorphism of $H$; the existence criterion is unchanged.
\end{proof}

If $\mu$ belongs to a chosen maximal torus, pointwise condition~(1) is equivalent to $w\mu=-\mu$ for some Weyl element $w$.  Invariant polynomials give quick obstructions: an odd-degree invariant polynomial nonzero at $\mu$ excludes condition~(1).

\begin{corollary}[The $\mathrm{SU}(2)$ criterion]\label{cor:su2criterion}
Let $G=\mathrm{SU}(2)$ and $u\ne0$.  The reduction $Q$ is a circle bundle.  If $L=Q\times_{S^1}\C$ is its weight-one line bundle, a gauge transformation satisfying~\eqref{eq:signgauge} exists exactly when
\begin{equation}\label{eq:twoe}
 2c_1(L)=0\quad\text{in }H^2(M;\Z).
\end{equation}
\end{corollary}

\begin{proof}
The Weyl representative of $\mathrm{SU}(2)$ exchanges the two eigenlines and acts on $H=S^1$ by inversion.  Thus $Q^{\vartheta_n}$ is the inverse circle bundle, with first Chern class $-c_1(L)$.  Principal circle bundles are classified by their integral first Chern class, so Theorem~\ref{thm:signcriterion} gives~\eqref{eq:twoe}.
\end{proof}

The criterion concerns the field $u$.  For the pair $(A,u)$ with a fixed connection, preservation of $A$ imposes a further holonomy condition in Theorem~\ref{thm:holonomy}.

\begin{corollary}[Regular fields and the Weyl lattice]\label{cor:weyllattice}
Suppose $u$ is regular, so that $H=T$ is a maximal torus, and let $\Lambda=\pi_1(T)$.  Write $c(Q)\in H^2(M;\Lambda)$ for the class of the torus bundle $Q$.  If $w\in W(G,T)$ sends $\mu$ to $-\mu$, then
\[
 u\text{ admits a gauge sign realization}
 \quad\Longleftrightarrow\quad c(Q)=w_*c(Q).
\]
For regular $\mu$ the element $w$, when it exists, is unique and has order two.
\end{corollary}

\begin{proof}
Principal $T$-bundles are classified by $H^2(M;\Lambda)$, and twisting the $T$ action by a representative of $w$ sends the class to $w_*c(Q)$.  Theorem~\ref{thm:signcriterion} gives the equivalence.  The Weyl group acts freely on the regular points of a fixed Cartan algebra, so $w^2\mu=\mu$ implies $w^2=1$.
\end{proof}

\section{Fixed-torus descent and the gauge normalizer}

Fix a maximal torus $T\subset G$, with Weyl group $W=N_G(T)/T$ and cocharacter lattice $\Lambda=\pi_1(T)$. Let $Q\to M$ be a principal $T$ bundle, $P=Q\times_TG$, and $c=c(Q)\in H^2(M;\Lambda)$. For $w\in W$, the notation $w_*c$ uses its conjugation action on $\Lambda$. Define
\[
 W_Q=\{w\in W:w_*c=c\},\qquad
 S=T^{W_Q}=\{t\in T:w(t)=t\text{ for every }w\in W_Q\}.
\]
The group $S$ need not be connected. Denote by $N_{W_Q}$ the inverse image of $W_Q$ in $N_G(T)$.

\begin{theorem}[Fixed-torus descent]\label{thm:descent}
The $T$ bundle $Q$ admits a reduction to $S=T^{W_Q}$. More generally, if a finite subgroup $\Gamma\leq W$ fixes $c(Q)$, then $Q$ reduces to $T^\Gamma$.
\end{theorem}

\begin{proof}
For the second assertion form the homomorphism of tori
\[
 \Phi:T\longrightarrow\prod_{w\in\Gamma}T,
 \qquad \Phi(t)=(w(t)t^{-1})_{w\in\Gamma}.
\]
Its image $D$ is a connected closed subtorus of the product, and its kernel is the full fixed subgroup $T^\Gamma$, including its finite components. Extend $Q$ along $T\to D$. The characteristic class $d\in H^2(M;\pi_1D)$ of this $D$ bundle maps, under $D\hookrightarrow\prod_{w\in\Gamma}T$, to $((w_*-1)c)_{w\in\Gamma}=0$.

The cocharacter lattice of a connected subtorus is a primitive sublattice of the ambient free lattice: the quotient is the cocharacter lattice of the quotient torus and is free. Its inclusion therefore has a $\Z$-linear retraction. Applying $H^2(M;-)$ shows that the displayed map on characteristic classes is injective. Hence $d=0$, and the extended $D$ bundle is trivial. A section of it is equivalent to a $T^\Gamma$ reduction of $Q$. Such a section can be smoothed when $M$ is smooth. Take $\Gamma=W_Q$ for the first assertion.
\end{proof}

Define the normalizing gauge group of $Q$ by
\[
 \mathcal N_Q=\{F\in\Gau(P): F(Q)=Qn\text{ for some }n\in N_G(T)\}.
\]
On a connected base, the Weyl component of a normalizing gauge is constant. Thus there is an exact sequence
\begin{equation}\label{eq:gauge-extension}
 1\longrightarrow C^\infty(M,T)\longrightarrow\mathcal N_Q
 \longrightarrow W_Q\longrightarrow1.
\end{equation}
For the last equality, a gauge over $w$ gives an isomorphism $Q\simeq Q^w$, and principal torus bundles are classified by $H^2(M;\Lambda)$; conversely such an isomorphism extends to $P$.

\begin{theorem}[Normalizer rigidity]\label{thm:pushout}
For a choice of the $S$ reduction in Theorem~\ref{thm:descent}, there is a group isomorphism
\begin{equation}\label{eq:pushout}
 \mathcal N_Q\simeq
 \bigl(C^\infty(M,T)\rtimes N_{W_Q}\bigr)/\Delta T,
 \qquad
 \Delta T=\{(\underline{t}^{-1},t):t\in T\},
\end{equation}
where $N_{W_Q}$ acts on maps $M\to T$ through its Weyl component and $\underline t$ is the constant map. In particular, the extension~\eqref{eq:gauge-extension} is the pushout of $1\to T\to N_{W_Q}\to W_Q\to1$ along the inclusion of constant maps $T\hookrightarrow C^\infty(M,T)$.

If $w\in W_Q$ has order $r$, then the least possible finite order of a gauge lift of $w$ equals the least possible order of an element in the coset $N_G(T)_w$. Moreover, the entire $W_Q$ acts by a group of gauge transformations if and only if $N_{W_Q}\to W_Q$ splits.
\end{theorem}

\begin{proof}
Write $P=R\times_SG$ using an $S$ reduction $R\subset Q$. Every $n\in N_{W_Q}$ centralizes $S$, because its Weyl element fixes each $s\in S$. Thus
\[
 J_n([q,g])=[q,ng],\qquad q\in R,
\]
is a well-defined gauge transformation, and $n\mapsto J_n$ is a homomorphism. Every $a\in C^\infty(M,T)$ similarly defines $A_a([q,g])=[q,a(\pi(q))g]$. These gauges satisfy $J_n A_aJ_n^{-1}=A_{w(a)}$ for the Weyl element $w$ of $n$. They generate $\mathcal N_Q$: if $F$ lies over $w$, choose $n$ over $w$; then $FJ_n^{-1}$ lies in the kernel of~\eqref{eq:gauge-extension} and equals some $A_a$. The kernel of $(a,n)\mapsto A_aJ_n$ is exactly the displayed $\Delta T$. This proves~\eqref{eq:pushout}.

Choose $n$ of least order in the coset over $w$. Then $J_n$ has that same order. Conversely, for any gauge $F$ over $w$ and any $q\in Q_x$, there is a unique $n_q\in N_G(T)_w$ with $F(q)=qn_q$. Equivariance and induction give $F^k(q)=qn_q^k$; hence the order of $F$ is at least the least order of a pointwise normalizer lift. Finally, a splitting $W_Q\to N_{W_Q}$ yields a splitting through $J$. If $W_Q\to\mathcal N_Q$ splits, evaluation at a fixed $q\in Q_x$ gives a homomorphism $W_Q\to N_{W_Q}$, proving the converse.
\end{proof}

\begin{remark}[The cyclic obstruction]\label{rem:norm}
Fix $w\in W_Q$ of order $r$ and a lift $F\in\mathcal N_Q$. Write $\mathcal A=C^\infty(M,T)$. Then $F^r\in\mathcal A^w$, and replacing $F$ by $aF$ replaces its $r$th power by $N_w(a)F^r$, where $N_w(a)=a\,w(a)\cdots w^{r-1}(a)$. Thus the obstruction to an order-$r$ lift is the class of $F^r$ in $\mathcal A^w/N_w\mathcal A=H^2(\mathbb Z/r;\mathcal A)$. Theorem~\ref{thm:pushout} identifies this class with the image of the pointwise normalizer class under constant maps. Evaluation at a base point is a left inverse on these cohomology groups. In particular, bundle twisting can remove $w$ from $W_Q$, while it cannot introduce an additional cyclic order anomaly once $w\in W_Q$.
\end{remark}

\begin{corollary}[Compatible symmetry-preserving connections]\label{cor:connection}
For any finite $\Gamma\leq W_Q$ there is a connection on $P$ preserving the original torus reduction and preserved by every gauge $J_n$ with $n\in N_\Gamma$. If $T^\Gamma$ is finite, this connection is flat. For a prescribed connection $A$ on $Q$, a gauge over $w$ that preserves $A$ exists precisely when its holonomy is contained in $T^w$.
\end{corollary}

\begin{proof}
Choose any connection on the $T^\Gamma$ reduction and extend its Lie-algebra-valued form to $P$. Each $J_n$ acts on this connection through $\Ad_n$; since $n$ centralizes $T^\Gamma$, the form is unchanged. If $T^\Gamma$ is finite, its Lie algebra is zero and the extended connection is flat. For the final assertion, a connection-preserving gauge is determined by its value $n$ at one frame, and parallel transport makes it global exactly when $n$ centralizes the holonomy group. Holonomy lies in $T$; an $n$ over $w$ centralizes it exactly when every holonomy element is fixed by $w$.
\end{proof}

\begin{corollary}[Symplectic exclusion for elliptic Weyl symmetry]\label{cor:elliptic}
Let $u$ be a parallel regular field of value $\mu\in\mathfrak t$ for a
connection preserving $Q$, and suppose its paired curvature
$\beta=B(u,F_A)$ is symplectic on a closed base of positive dimension.
If $\Gamma\leq W_Q$, then the orthogonal projection of $\mu$ to
$\mathfrak t^\Gamma$ is nonzero. In particular, $T^\Gamma$ cannot be
finite. When $G$ is semisimple, one cannot have $W_Q=W$.
\end{corollary}

\begin{proof}
The real characteristic class of $Q$ belongs to the fixed subspace
$H^2(M;\mathbb R)\otimes\mathfrak t^\Gamma$ because every element of
$\Gamma$ fixes $c(Q)$. The de Rham class of $\beta$ is its pairing with
$\mu$, up to the fixed Chern--Weil normalization. If the projection of
$\mu$ to $\mathfrak t^\Gamma$ vanished, then $[\beta]=0$. A symplectic
form on a closed positive-dimensional manifold has nonzero cohomology
class, as its top power has nonzero integral. The finite-fixed-group
assertion follows from $\mathfrak t^\Gamma=0$. For semisimple $G$ the
fixed subspace $\mathfrak t^W$ is zero.
\end{proof}

\section{Contact geometry forces a mirror obstruction and transgression}\label{sec:contact}

\begin{theorem}[Paired-curvature sign obstruction]\label{thm:betaobstruction}
Fix a nonzero parallel adjoint field $u$.  Among connections $A$ satisfying $\dd_Au=0$, the cohomology class
\[
 [\beta_A]\in H^2(M;\R),\qquad \pi^*\beta_A=B(u,F_A),
\]
is independent of $A$.  If a gauge transformation over $\id_M$ carries $u$ to $-u$, then $[\beta_A]=0$.  Consequently a symplectic $\beta_A$ on a closed base excludes such a gauge transformation for every compact connected semisimple $G$.
\end{theorem}

\begin{proof}
Let $A_1=A_0+a$ preserve $u$ along with $A_0$.  The connection difference $a$ is horizontal and equivariant, and $[a,u]=0$.  Invariance of $B$ and $\dd_{A_0}u=0$ give
\[
 B(u,F_{A_1})-B(u,F_{A_0})
 =\dd B(u,a)+\tfrac12B(u,[a\wedge a])
 =\dd B(u,a).
\]
The one-form $B(u,a)$ is basic, proving independence of the class.

Now let $F$ be a gauge transformation with $u\circ F=-u$.  Put $A'=F^*A$.  Pulling back $\dd_Au=0$ shows that $A'$ also preserves $u$.  Since $F$ covers $\id_M$,
\[
 B(u,F_{A'})=B(u,F^*F_A)
 =-F^*B(u,F_A)=-\pi^*\beta_A.
\]
Thus $\beta_{A'}=-\beta_A$, while connection independence gives $[\beta_{A'}]=[\beta_A]$.  Hence $[\beta_A]=0$.  A symplectic form on a closed base has nonzero de Rham class by its volume integral.
\end{proof}

For $G=\mathrm{SU}(2)$, the defining complex rank-two bundle associated to $P$ splits along $Q$ as
\begin{equation}\label{eq:eigensplit}
 E=L\oplus L^{-1}.
\end{equation}
We use the integral class $e=c_1(L)$.  The sign-realization criterion for $u$ involves the torsion of $e$; the contact condition detects its real part.

\begin{theorem}[Contact sign obstruction and fiber transgression]\label{thm:contacttopology}
Let $M^{2m}$ be closed and connected, $m\ge1$, and let $(P,A,u)$ be a parallel compatible pair for $G=\mathrm{SU}(2)$.  If $\alpha=B(u,A)$ is contact, then:
\begin{enumerate}
\item $e_\R\in H^2(M;\R)$ is a symplectic class.  No gauge transformation carries $u$ to $-u$.
\item If $m\ge2$, then $c_2(E)=-e^2$ is nonzero in real cohomology.  The pullback
\begin{equation}\label{eq:H3iso}
 \pi^*:H^3(M;\R)\xrightarrow{\ \cong\ }H^3(P;\R)
\end{equation}
is an isomorphism.  The degree-three generator of $H^*(\mathrm{SU}(2);\R)$ has nonzero transgression.
\end{enumerate}
More generally, for every $k$ there is a short exact sequence of real vector spaces
\begin{equation}\label{eq:gysin}
\begin{gathered}
0\longrightarrow
 \operatorname{coker}\bigl(e^2\smile:H^{k-4}(M)\to H^k(M)\bigr)
 \longrightarrow H^k(P)\\
 \longrightarrow
 \ker\bigl(e^2\smile:H^{k-3}(M)\to H^{k+1}(M)\bigr)
 \longrightarrow0.
\end{gathered}
\end{equation}
Negative-degree groups are zero; coefficients in~\eqref{eq:gysin} are $\R$.
\end{theorem}

\begin{proof}
Choose a generator $T$ of the circle Lie algebra $\h=\g_\mu$ such that $L$ is the weight-one associated line.  Since $u|_Q=\mu$ is a nonzero multiple of $T$, the restriction $A_Q$ is a circle connection and~\eqref{eq:beta} gives $\beta$ as a nonzero real multiple of its curvature.  Chern--Weil theory therefore gives $[\beta]=a e_\R$ for some $a\ne0$.  Proposition~\ref{prop:levi} makes $\beta$ symplectic.  Its volume integral shows
\[
 e_\R^m=a^{-m}[\beta]^m\ne0.
\]
In particular $2e\ne0$, so Corollary~\ref{cor:su2criterion} excludes a gauge sign realization.

For $m\ge2$, $e^2\ne0$.  The Whitney formula applied to~\eqref{eq:eigensplit} yields $c_2(E)=-e^2$~\cite{MilnorStasheff}.  The principal $\mathrm{SU}(2)$ bundle $P$ is the unit $S^3$-bundle of $E$: a unit vector in a Hermitian determinant-one two-plane determines a unique special-unitary frame.  Its real Euler class is $c_2(E)$.  The Gysin sequence of this oriented sphere bundle gives~\eqref{eq:gysin}~\cite{BottTu}.  In degree three the map $H^0(M;\R)\xrightarrow{\smile e^2}H^4(M;\R)$ is injective, and the preceding degree is zero.  This proves~\eqref{eq:H3iso}.  Equivalently, in the Serre spectral sequence~\cite{Serre} the fiber's degree-three generator maps under $d_4$ to a nonzero multiple of $c_2(E)$.
\end{proof}

The sign obstruction for the field is global.  Curvature also obstructs a gauge symmetry of the hyperplane itself.

\begin{proposition}[Coorientation rigidity under gauge maps]\label{prop:coorientation}
Let $(P,A,u)$ be a parallel compatible pair for any compact connected semisimple $G$.  If $\beta\ne0$, no gauge transformation $F:P\to P$ preserves $\mathcal D$ while reversing its coorientation.
\end{proposition}

\begin{proof}
Such an $F$ would satisfy $F^*\alpha=-r\alpha$ for a positive function $r$.  On a fundamental vector $a^\#$, gauge equivariance gives
\[
 B(u(Fp),a)=(F^*\alpha)_p(a^\#)
 =-r(p)B(u(p),a).
\]
The norm of $u$ is constant, so $r=1$ and $u(Fp)=-u(p)$.  Choose $n$ and the semilinear map $f:Q\to Q$ from the proof of Theorem~\ref{thm:signcriterion}.  The one-form $\alpha$ is invariant under right translation by $n$, hence on $Q$ we have $f^*\alpha_Q=(F|_Q)^*\alpha=-\alpha_Q$, where $\alpha_Q=\alpha|_Q$.  Because $H$ fixes $\mu$, its Lie algebra centralizes $\mu$; thus $\alpha_Q=B(\mu,A_Q)$ and $\dd\alpha_Q=\pi_Q^*\beta$.  Since $f$ covers $\id_M$, differentiating $f^*\alpha_Q=-\alpha_Q$ gives $\pi_Q^*\beta=-\pi_Q^*\beta$.  Pullback by the submersion $\pi_Q$ is injective, so $\beta=0$, a contradiction.
\end{proof}

\section{The diagonalization bundle and its Weyl quotient}

The torus normalizer also controls the bundle used by Friedman and
Park~\cite[Section~2]{FriedmanPark2016} to compare two multiplicity-free
normal matrices. Let $T=\U(1)^n\subset\U(n)$, $W=S_n$, and
$\mathcal F_n=\U(n)/T$. A point of $\mathcal F_n$ is an ordered family
$(P_1,\ldots,P_n)$ of mutually orthogonal rank-one projections summing to
$I$. Friedman and Park's space $\mathcal B_n$ consists of two
\emph{unordered} such families together with a bijection between them;
their bundle $p:\mathcal E_n\to\mathcal B_n$ has as its fiber the
unitary operators implementing the indicated bijection.  They ask for
$H^2(\mathcal B_n;\Z)$ and for a structure group and associated principal
bundle for $p$ in~\cite[Section~8]{FriedmanPark2016}.  We also compute
the obstruction group with its natural permutation local coefficients.
Put
$N=N_{\U(n)}(T)=T\rtimes W$ and
\[
 K=N\mathbin{\times_W}N
   =\{(n_1,n_2)\in N^2:\bar n_1=\bar n_2\in W\},
 \qquad \Delta T=\{(t,t):t\in T\}\subset K.
\]

\begin{theorem}[Topology and structure group of the diagonalization bundle]
\label{thm:diagonalization-base}
For $n\geq2$ there are natural homeomorphisms
\[
 \mathcal B_n\simeq(\mathcal F_n\times\mathcal F_n)/W
 \simeq(\U(n)\times\U(n))/K.
\]
Its fundamental group is $S_n$ and
$H^2(\mathcal B_n;\Z)\simeq\Z/2$.
The generator is the first Chern class of the complex line associated
to the sign representation of the covering
$\mathcal F_n\times\mathcal F_n\to\mathcal B_n$.

The normal subgroup $\Delta T$ is the kernel of the affine action
\[
 K\longrightarrow\operatorname{Aff}_W(T)=T\rtimes W,
 \qquad (n_1,n_2)\cdot z=n_2zn_1^{-1},
\]
and the displayed map induces $K/\Delta T\simeq T\rtimes W$.
Consequently $(\U(n)\times\U(n))/\Delta T\to\mathcal B_n$
is a principal $T\rtimes W$ bundle, and $\mathcal E_n$ is its
associated $T$ bundle under the affine action. For $n=1$ the base is a
point and its positive-degree cohomology vanishes.
\end{theorem}

\begin{proof}
Ordering the first projection family and transporting that order
through its bijection orders the second family. The $n!$ choices differ
by a simultaneous permutation, giving the first quotient. The source
is compact and $\mathcal B_n$ is Hausdorff in its metric topology, so
the induced continuous bijection is a homeomorphism. The second
quotient follows by lifting both ordered families to unitary frames.

The long exact sequence for $T\to\U(n)\to\mathcal F_n$ shows
$\pi_1(\mathcal F_n)=0$: the map
$\pi_1(T)=\Z^n\to\pi_1(\U(n))=\Z$ is summation and is surjective.
Thus $\mathcal F_n^2$ is the universal cover of $\mathcal B_n$,
and $H_1(\mathcal B_n;\Z)=(S_n)_{\mathrm{ab}}=\Z/2$.
The degree-two cohomology of the flag manifold is
\[
 H^2(\mathcal F_n;\mathbb Q)
 \simeq\mathbb Q^n/\mathbb Q(1,\ldots,1),
\]
with $S_n$ permuting coordinates. Its invariant subspace is zero for
$n\geq2$. Since $H^1(\mathcal F_n;\mathbb Q)=0$, the same holds for
$H^2(\mathcal F_n^2;\mathbb Q)^W$. Transfer for the finite covering
gives $H^2(\mathcal B_n;\mathbb Q)=0$. The base is a compact manifold
and has the homotopy type of a finite CW complex. Hence
$\operatorname{Hom}(H_2(\mathcal B_n;\Z),\Z)=0$, and the integral
universal coefficient sequence yields
\[
 H^2(\mathcal B_n;\Z)
 \simeq\operatorname{Ext}(H_1(\mathcal B_n;\Z),\Z)
 \simeq\Z/2.
\]
The Bockstein of the sign class in $H^1(\mathcal B_n;\Z/2)$ is nonzero
because $H^1(\mathcal B_n;\Z)=0$; it is the first Chern class stated
above.

The right $K$ action on $\U(n)^2$ is
$(u,v)\cdot(n_1,n_2)=(un_1,vn_2)$. If $z\in T$, the unitary operator
$vzu^{-1}$ maps the $i$th line of the first frame to its matched line
of the second. The associated-bundle relation
$[(un_1,vn_2),z]=[(u,v),n_2zn_1^{-1}]$ gives the model-fiber action
$z\mapsto n_2zn_1^{-1}$. Since $n_1,n_2$ have the same
Weyl component, $n_2n_1^{-1}\in T$ and this is an affine map of $T$.
It fixes every $z$ exactly when $n_1=n_2\in T$, proving the kernel
claim. Permutation matrices split $N\to W$, and every translation of
$T$ occurs, so the image is all $T\rtimes W$. The associated-bundle
description follows.
\end{proof}

Let $\boldsymbol\Lambda$ be the local system on $\mathcal B_n$ with
fiber $\Z^n$ and permutation monodromy. The affine-torus bundle has a
characteristic class
$\Theta_n\in H^2(\mathcal B_n;\boldsymbol\Lambda)$;
since a torus is a $K(\Z^n,1)$, this is the complete obstruction to a
section. Pullback along the classifying map of a pair of normal
matrices is Friedman--Park's unitary-equivalence obstruction
\cite[Theorem~3.2]{FriedmanPark2016}.
Each of the two universal unordered flags also defines a principal
$T\rtimes W$ bundle; its associated affine-torus torsor has a class in
the same cohomology group, called its twisted first-Chern vector below.

\begin{proposition}[The universal obstruction group]
\label{prop:universal-obstruction}
For $n\geq2$,
\[
 H^2(\mathcal B_n;\boldsymbol\Lambda)\simeq
 \begin{cases}
 \Z^2,&n=2,\\
 \Z^2\oplus\Z/2,&n\geq3.
 \end{cases}
\]
The two free generators can be chosen as the twisted first-Chern
vectors of the two universal unordered flags; $\Theta_n$ is their
difference.
\end{proposition}

\begin{proof}
Use the Cartan--Leray spectral sequence of the free $W$ cover
$\mathcal F_n^2\to\mathcal B_n$ with coefficients $\Z^n$.
The $q=1$ row vanishes. In total degree two, the group-cohomology term
is, by Shapiro's lemma,
\[
 H^2(W;\Z^n_{\mathrm{perm}})
 \simeq H^2(S_{n-1};\Z)
 =\begin{cases}0,&n=2,\\\Z/2,&n\geq3.\end{cases}
\]
Moreover $H^2(\mathcal F_n;\Z)=\Z^n/\Z(1,\ldots,1)$.
A $W$-invariant element of
$H^2(\mathcal F_n;\Z)\otimes\Z^n$
is equivalently a $W$-equivariant map
$\Z^n\to\Z^n/\Z(1,\ldots,1)$. Such a map is an integral multiple
of $e_i\mapsto[e_i]$: the image of $e_1$ must be fixed by its
stabilizer $S_{n-1}$, whose fixed subgroup is generated by $[e_1]$.
Each of the two flag factors therefore gives one primitive integral
invariant generator. Both lift to twisted classes on $\mathcal B_n$, because
each unordered flag carries its actual principal $N=T\rtimes W$
bundle. Thus the only possible outgoing differential from this
invariant term vanishes. The free quotient of the degree-two filtration
is generated by the restrictions of these two classes, and the group is an extension
of $\Z^2$ by the displayed group-cohomology term. It splits as an
abelian-group extension. In local frame coordinates the transition
phases for the affine comparison bundle are the second flag phases
times the inverses of the first; hence its class is the difference of
the two twisted first-Chern vectors.
\end{proof}

\section{Singular fields and nonabelian twists}

Let $\mu\in\mathfrak g$ be singular and nonzero, $H=G_\mu$, and $Q=u^{-1}(\mu)$. The connection $A$ restricts to $H$. A second value $\nu$ defines a parallel field from this same reduction by
\begin{equation}\label{eq:second-field}
 u_\nu(qg)=\Ad_{g^{-1}}\nu
\end{equation}
provided $\nu$ is fixed by $H$. In particular, for $n\in N_G(H)$, the value $\nu=\Ad_{n^{-1}}\mu$ is admissible and has stabilizer $H$. An arbitrary Weyl permutation of a singular value need not belong to $N_G(H)$ and need not define~\eqref{eq:second-field}. For example, the $2+1$ stabilizer $S(\U(2)\times\U(1))$ in $\SU(3)$ is self-normalizing.

For $n\in N_G(H)$ set $\theta_n=\Ad_n|_H$ and twist the right action of $Q$ to $q\star h=q\theta_n^{-1}(h)$, obtaining $Q^{\theta_n}$.  Define
\[
 \begin{aligned}
 \mathcal N_Q^H&=\{F\in\Gau(P):F(Q)=Qn\text{ for some }n\in N_G(H)\},\\
 \mathcal W_Q^H&=\{nH\in N_G(H)/H:Q\simeq Q^{\theta_n}\}.
 \end{aligned}
\]

\begin{proposition}[Singular transporter criterion]\label{prop:singular}
There is a gauge transformation carrying $u_\mu$ to $u_{\Ad_{n^{-1}}\mu}$ if and only if $Q\simeq Q^{\theta_n}$ as $H$ bundles. Equivalently, the classes $[Q]$ and $\theta_{n*}[Q]$ agree in the nonabelian bundle-class set $[M,BH]$. The normalizing gauge transformations form an extension
\begin{equation}\label{eq:nonabelian-extension}
 1\longrightarrow\Gau(Q)\longrightarrow\mathcal N_Q^H
 \longrightarrow\mathcal W_Q^H\longrightarrow1,
 \qquad \mathcal W_Q^H\subset N_G(H)/H.
\end{equation}
If $nH$ has order $r$ and $F$ is a lift, then $F^r$ lies in $\Gau(Q)$. For $h\in\Gau(Q)$,
\begin{equation}\label{eq:nonabelian-norm}
 (hF)^r=h\,\phi(h)\cdots\phi^{r-1}(h)F^r,
 \qquad\phi=\operatorname{Ad}_F|_{\Gau(Q)}.
\end{equation}
\end{proposition}

\begin{proof}
For $q\in Q$, a gauge $F$ with $F^*u_\mu=u_{\Ad_{n^{-1}}\mu}$ has $F(q)\in Qn$. Write $F(q)=f(q)n$. Equivariance gives $f(qh)=f(q)\theta_n(h)$, so $f:Q^{\theta_n}\to Q$ is an $H$-bundle isomorphism. Conversely such an $f$ extends by $F(qg)=f(q)ng$. The same argument with any normalizer coset gives~\eqref{eq:nonabelian-extension}, whose kernel consists of gauges of $Q$. Formula~\eqref{eq:nonabelian-norm} is the multiplication rule for the extension. The cohomology set $[M,BH]$ is a pointed set; there is in general no subtraction of its classes into an abelian obstruction group.
\end{proof}

\begin{remark}
Theorem~\ref{thm:descent} uses the single class of a torus bundle in $H^2(M;\Lambda)$. A nonabelian $H$ bundle can retain higher characteristic and unstable information. Its lift-power equation~\eqref{eq:nonabelian-norm} is a nonabelian norm equation; $F^r$ may vary over the base and need not be central.
\end{remark}

\begin{theorem}[A nullhomotopic singular map without global conjugation]\label{thm:s4}
There is a smooth nullhomotopic map $g:S^4\to\SU(4)$ with constant singular spectrum such that no smooth map $h:S^4\to\SU(4)$ conjugates $g$ into a fixed maximal torus. Its positive and negative rank-two eigenbundles have second Chern classes $+1$ and $-1$. Equivalently, a trivial $\SU(4)$ bundle admits a parallel singular field that is pointwise conjugate to its sign reversal but has no sign gauge.
\end{theorem}

\begin{proof}
Let $E_+$ and $E_-$ be rank-two $\SU(2)$ bundles over $S^4$ with $c_2(E_+)=1$ and $c_2(E_-)=-1$. Their sum $V$ has $c_1(V)=c_2(V)=0$. Since $\SU(4)$ bundles over $S^4$ are classified by $c_2$, choose a unitary determinant-one trivialization $V\simeq S^4\times\C^4$. For $0<t<\pi$, let $g$ act as $e^{it}$ on $E_+$ and $e^{-it}$ on $E_-$. It is $\SU(4)$ valued and has two distinct eigenvalues of multiplicity two. It is nullhomotopic because $\pi_4(\SU(4))=0$.

If $hgh^{-1}$ were torus valued, its diagonal entries would be continuous functions taking values in the two-element set $\{e^{it},e^{-it}\}$, hence constant on $S^4$. The corresponding two eigenbundles would then be trivial, contradicting their second Chern classes. A block connection makes $u=i(1_{E_+},-1_{E_-})$ parallel; a sign gauge would identify $E_+$ with $E_-$ and is likewise impossible.
\end{proof}

\begin{remark}
The orbit $\SU(4)/S(\U(2)\times\U(2))$ retains the class: its fourth homotopy group maps to $\ker(\pi_3(H)\to\pi_3(\SU(4)))\simeq\ker(\Z^2\xrightarrow{+}\Z)\simeq\Z$. The map to $\SU(4)$ loses it. Blau and Thompson~\cite[Section~6]{BlauThompson} discuss the corresponding singular global-diagonalization problem.
\end{remark}

\section{Four-dimensional spectral realization}

The second-Chern accounting used below also gives an exact realization
criterion for the three-eigenline question in the latter part of
Friedman and Park's Example~2.5~\cite{FriedmanPark2024}.
Here $X$ is a connected finite CW complex of dimension at most four,
and all bundles and matrices are continuous. Fix continuous functions
$\lambda_1,\ldots,\lambda_s:X\to\C$ satisfying
$\lambda_i(x)\ne\lambda_j(x)$ for every $x\in X$ and $i\ne j$,
and positive integers
$m_1,\ldots,m_s$ with $\sum_i m_i=n\geq2$. Eigenvalues are labelled by
these functions, and each has constant multiplicity $m_i$.

\begin{theorem}[Four-dimensional spectral realization]
\label{thm:four-dimensional-realization}
Unitary-equivalence classes of normal maps $A:X\to M_n(\C)$ with
characteristic polynomial
$\prod_i(t-\lambda_i)^{m_i}$ are in bijection with tuples
$(a_i,b_i)_{i=1}^s$, where
$a_i\in H^2(X;\Z)$ and $b_i\in H^4(X;\Z)$, subject to
\[
 b_i=0\quad\text{if }m_i=1,\qquad
 \sum_i a_i=0,\qquad
 \sum_i b_i+\sum_{i<j}a_i a_j=0.
\]
The tuple records $(c_1(E_i),c_2(E_i))$ of the labelled eigenbundles.
On a closed smooth four-manifold, smooth eigenvalue functions admit
smooth realizations.
\end{theorem}

\begin{proof}
The spectral projections of a normal matrix with separated,
constant-multiplicity eigenvalues give an orthogonal decomposition
$\underline{\C}^{\,n}=\bigoplus_i E_i$. Whitney multiplication proves
the displayed conditions. Conversely, line bundles are classified by
$c_1$, and complex bundles of rank at least two on a four-complex are
classified by $(c_1,c_2)$, with every pair realized. For the latter
fact, the four-type of $B\U(m)$, $m\geq2$, has homotopy groups $\Z$ in
degrees two and four and no others in positive degrees at most four;
its possible $k$-invariant lies in
$H^5(K(\Z,2);\Z)=0$. The Chern classes give the resulting
coordinates. Choose Hermitian bundles $E_i$ of the prescribed classes.
Their sum $V$ has $c_1(V)=c_2(V)=0$, so the same classification in rank
$n$ makes $V$ trivial. A fiberwise polar decomposition makes a
trivialization unitary. Transporting the orthogonal summands into
$X\times\C^n$ gives projections $P_i$ and the required normal map
$A=\sum_i\lambda_iP_i$.

If two maps have the same tuple, their corresponding eigenbundles are
isomorphic. Polar decomposition makes those isomorphisms unitary, and
their direct sum conjugates the two maps. Conversely, a unitary
conjugacy preserves every labelled eigenbundle and its Chern classes.
Standard smoothing of bundles, metrics, and trivializations gives the
last assertion.
\end{proof}

\begin{corollary}[The three-eigenline problem on $S^2\times S^2$]
\label{cor:friedman-park}
Let $x,y$ generate $H^2(S^2\times S^2;\Z)$, with $x^2=y^2=0$.
For any fixed multiplicity-free characteristic polynomial of degree
three, every integer quadruple $(k_1,k_2,\ell_1,\ell_2)$ satisfying
\[
 k_1(2\ell_1+\ell_2)+k_2(\ell_1+2\ell_2)=0
\]
is realized by a normal matrix whose first two eigenlines have
$c_1=k_1x+\ell_1y$ and $c_1=k_2x+\ell_2y$; the third has the negative
sum. These and only these quadruples occur, and each gives one
unitary-equivalence class for the chosen ordering of roots.
\end{corollary}

\begin{proof}
Put $a_1=k_1x+\ell_1y$, $a_2=k_2x+\ell_2y$, and
$a_3=-a_1-a_2$. Then
\[
 \sum_{i<j}a_i a_j
 =-\bigl[k_1(2\ell_1+\ell_2)
          +k_2(\ell_1+2\ell_2)\bigr]xy.
\]
Theorem~\ref{thm:four-dimensional-realization} gives both directions
and uniqueness. For an explicit parametrization, put
$p=2k_1+k_2$ and $q=k_1+2k_2$. If $(k_1,k_2)=(0,0)$, both
$\ell_1,\ell_2$ are arbitrary. Otherwise, with
$d=\gcd(|p|,|q|)>0$, all solutions are
$(\ell_1,\ell_2)=t(q/d,-p/d)$ for $t\in\Z$.
The roots of a multiplicity-free polynomial over this simply
connected base can be labelled globally.
\end{proof}

\begin{remark}[A sharp unstable boundary]
The four-dimensional hypothesis matters for repeated eigenvalues.
Let $E\to S^5$ be the nontrivial complex rank-two bundle classified
by $\pi_4(\U(2))\simeq\Z/2$. Its Chern classes vanish, while
$E\oplus\underline{\C}^{\,2}$ is trivial because
$\pi_4(\U(4))=0$. Write $P_E$ and $P_0$ for its orthogonal rank-two
projections and choose distinct constants $\lambda,\mu$. The normal
maps $A=\lambda P_E+\mu P_0$ and
$B=\mu P_E+\lambda P_0$ have identical eigenbundle
Chern classes and are pointwise conjugate. A global unitary conjugacy
would identify $E$ with $\underline{\C}^{\,2}$, which is impossible.
This is an unstable torsion obstruction beyond the four-type.
\end{remark}

\section{Identical symplectic curvature with inequivalent singular reductions}

Let $h=c_1(\mathcal O(1))\in H^2(\mathbb{CP}^2;\Z)$ and choose a Fubini--Study form $\omega$ representing $h$. The Euler sequence gives
\[
 c\bigl(T\mathbb{CP}^2\otimes\mathcal O(-1)\bigr)=1+h+h^2.
\]
Set
\[
 E=\mathcal O(1)\oplus\mathcal O,
 \qquad F=T\mathbb{CP}^2\otimes\mathcal O(-1),
 \qquad L=\mathcal O(-2).
\]
Then $\det E\simeq\det F\simeq\mathcal O(1)$, $c_2(E)=0$, $c_2(F)=h^2$, and $\det(E\oplus F\oplus L)$ is trivial. Let $P=\Fr_{\SU(5)}(E\oplus F\oplus L)$. Choose a block connection $A$ with determinant curvatures $-2\pi i\omega$ on both $E$ and $F$ and $4\pi i\omega$ on $L$.

\begin{theorem}[Curvature-blind nonabelian obstruction]\label{thm:cp2}
The $A$-parallel fields
\[
 u=i(1_E,0_F,-2_L),\qquad
 v=i(0_E,1_F,-2_L)
\]
are pointwise conjugate and have the same singular stabilizer. Their paired curvature forms agree and are symplectic:
\[
 B(u,F_A)=B(v,F_A)=-100\pi\omega,
 \qquad B=-\kappa\text{ on }\mathfrak{su}(5).
\]
Their determinant-bundle data are invariant under the block swap, yet no gauge transformation carries $u$ to $v$.
\end{theorem}

\begin{proof}
The two rank-two blocks can be exchanged at each point by a determinant-one unitary matrix, so the values are conjugate. Both fields are parallel because they are scalar on the blocks preserved by $A$. The prescribed determinant curvatures and $B(X,Y)=-10\operatorname{tr}(XY)$ yield the stated equality as differential forms: the weighted trace for either field is $10\pi\omega$.

The stabilizer is $H=S(\U(2)\times\U(2)\times\U(1))$, and the swap acts trivially on the determinant classes because $\det E\simeq\det F$. A global gauge carrying $u$ to $v$ would identify the eigenbundles $E$ and $F$. Their second Chern classes differ, so this is impossible. Moreover $F$ cannot split into two line bundles: if it were $\mathcal O(a)\oplus\mathcal O(b)$, the equalities $a+b=1$ and $ab=1$ would have no integer solution. Thus the example does not admit a refinement to a regular torus reduction.
\end{proof}

\begin{corollary}[The same ambient bundle and curvature class]\label{cor:sameP}
On one fixed principal $\SU(5)$ bundle over $\mathbb{CP}^2$ there are block connections and pairs of pointwise-conjugate parallel fields with the same paired-curvature form $-100\pi\omega$ for which a block-swap gauge exists, and other such pairs for which it does not.
\end{corollary}

\begin{proof}
For each integer $a$, a smooth rank-two bundle $E_a$ over $\mathbb{CP}^2$ with $c_1(E_a)=h$ and $c_2(E_a)=a h^2$ exists: starting from $E_1$, a local clutching modification in a four-ball changes $c_2$ by any integer while preserving $c_1$. Compare the decompositions
\[
 V_{\mathrm{yes}}=E_1\oplus E_1\oplus L,
 \qquad V_{\mathrm{no}}=E_0\oplus E_2\oplus L.
\]
Both have $c_1=0$ and $c_2=-h^2$. Principal $\SU(5)$ bundles over the four-complex $\mathbb{CP}^2$ are classified by $c_2$, so identify their frame bundles. Give every $E_a$ the same chosen determinant curvature $-2\pi i\omega$ and use the same connection on $L$. The calculation in Theorem~\ref{thm:cp2} gives the common paired form. The first decomposition admits an involutive block swap; the second does not because $c_2(E_0)\ne c_2(E_2)$.
\end{proof}

\begin{theorem}[An affine second-Chern obstruction on a trivial bundle]\label{thm:trivialP}
Let $(M^4,\omega)$ be a closed connected symplectic manifold with integral class $a=[\omega]\in H^2(M;\Z)$. For each integer $k$, there is a trivial principal $\SU(5)$ bundle with a block connection $A_k$ and pointwise-conjugate, $A_k$-parallel singular fields $u_k,v_k$ such that
\[
 B(u_k,F_{A_k})=B(v_k,F_{A_k})=-100\pi\omega.
\]
The determinant data of their rank-two eigenbundles agree. A gauge carrying $u_k$ to $v_k$ exists if and only if
\begin{equation}\label{eq:affine-parity}
 2k=3\langle a^2,[M]\rangle.
\end{equation}
In particular, when $\langle a^2,[M]\rangle$ is odd, no member of this infinite family has a block-swap gauge, although the ambient bundle is trivial and the paired forms are symplectic and identical.
\end{theorem}

\begin{proof}
Choose rank-two Hermitian bundles $E_k,F_k$ with
\[
 c_1(E_k)=c_1(F_k)=a,\qquad
 c_2(E_k)=k[M]^*,\quad
 c_2(F_k)=(3\langle a^2,[M]\rangle-k)[M]^*,
\]
where $[M]^*$ is the orientation generator of $H^4(M;\Z)$. Such bundles exist for every $k$: their first Chern class is fixed by a line bundle, while modification of an $\SU(2)$ clutching map on an embedded four-ball changes $c_2$ by any integer. Equivalently, the four-type of $B\U(2)$ is $K(\Z,2)\times K(\Z,4)$. Let $L$ have $c_1(L)=-2a$ and set $V_k=E_k\oplus F_k\oplus L$. The Whitney formula gives
\[
 c_1(V_k)=0,\qquad
 c_2(V_k)=c_2(E_k)+c_2(F_k)+a^2-4a^2=0.
\]
The four-type of $B\SU(5)$ is $K(\Z,4)$, so the determinant-trivialized $\SU(5)$ bundle of $V_k$ is trivial. Fix determinant connections on $E_k,F_k$ with the same curvature $-2\pi i\omega$ and on $L$ with curvature $4\pi i\omega$. A block connection and the two scalar block fields of Theorem~\ref{thm:cp2} give the asserted equality of forms.

A block-swap gauge exists exactly when $E_k\simeq F_k$. Rank-two bundles on a four-complex are classified by $(c_1,c_2)$, so this is equivalent to~\eqref{eq:affine-parity}. When the square of $a$ is odd, the right side is odd and the equation has no integer solution. The infinitely many choices of $k$ give distinct second-Chern charges of the first eigenbundle, although their ambient bundles and paired forms are fixed.
\end{proof}

\begin{remark}[Homotopy source of the charge]
For $H=S(\U(2)\times\U(2)\times\U(1))$, one has $\pi_3(H)\simeq\Z^2$, while extension to $\SU(5)$ adds the two components in $\pi_3(\SU(5))\simeq\Z$. Since $\pi_4(\SU(5))=0$, the homotopy exact sequence gives $\pi_4(\SU(5)/H)\simeq\ker(\Z^2\xrightarrow{(x,y)\mapsto x+y}\Z)\simeq\Z$. The block swap negates this kernel. Theorem~\ref{thm:trivialP} realizes the resulting affine charge lattice on a fixed trivial ambient bundle; equation~\eqref{eq:affine-parity} describes its fixed point.
\end{remark}

\section{An affine charge lattice and its Weyl stabilizers}

The two-block obstruction admits a uniform form. Fix an integer $q\geq2$, set
$N=2q+1$, and write $A=\langle a^2,[M]\rangle$. The affine lattice
\begin{equation}\label{eq:charge-lattice}
 \mathcal C_q(a)=\left\{(k_1,\ldots,k_q)\in\Z^q:
 \sum_{j=1}^q k_j=\frac{q(q+1)}2 A\right\}
\end{equation}
has its natural $S_q$ action by coordinate permutations.

\begin{theorem}[Realization and symmetry of the charge lattice]\label{thm:charge-lattice}
Let $(M^4,\omega)$ be a closed connected integral symplectic manifold,
$a=[\omega]$, and choose pairwise distinct real numbers
$\lambda_1,\ldots,\lambda_q$ with nonzero sum such that
$\lambda_0=-2\sum_j\lambda_j$ differs from every $\lambda_j$.
For every $\mathbf k\in\mathcal C_q(a)$ there is a block decomposition of a
\emph{trivial} $\SU(N)$ bundle,
\[
 \underline{\C}^{N}\simeq E_1\oplus\cdots\oplus E_q\oplus L,
 \quad \operatorname{rank}E_j=2,\quad
 c_1(E_j)=a,\quad c_2(E_j)=k_j[M]^*,\quad c_1(L)=-qa.
\]
It admits a block connection $A_{\mathbf k}$ whose determinant curvature
is $-2\pi i\omega$ on every $E_j$. The parallel fields
\[
 u_\sigma=i\bigl(\lambda_{\sigma(1)}1_{E_1},\ldots,
 \lambda_{\sigma(q)}1_{E_q},\lambda_0 1_L\bigr),
 \qquad \sigma\in S_q,
\]
are pointwise conjugate and have the same symplectic paired-curvature form
\begin{equation}\label{eq:charge-beta}
 B(u_\sigma,F_{A_{\mathbf k}})
 =-4\pi(2q+1)^2\left(\sum_j\lambda_j\right)\omega.
\end{equation}
Their common stabilizer is $H=S(\U(2)^q\times\U(1))$, with
$N_{\SU(N)}(H)/H\simeq S_q$. The image of the normalizing gauge group
in $S_q$ is exactly $\operatorname{Stab}_{S_q}(\mathbf k)$. This entire
stabilizer has a coherent lift to the gauge group, preserving a suitable
choice of $A_{\mathbf k}$ and permuting the full family
$\{u_\sigma:\sigma\in S_q\}$.

More generally, for $\Gamma\leq S_q$ with orbit sizes $d_1,\ldots,d_s$,
some point of $\mathcal C_q(a)$ admits a coherent $\Gamma$ gauge action
on its family of parallel fields if
and only if
\begin{equation}\label{eq:charge-gcd}
 \gcd(d_1,\ldots,d_s)\ \mid\ \frac{q(q+1)}2 A.
\end{equation}
In particular, a full $S_q$ gauge action occurs for some charge vector
exactly when $q$ is odd or $A$ is even.
The exact Weyl stabilizers arising in this family are precisely the
Young subgroups $S_{b_1}\times\cdots\times S_{b_t}\leq S_q$ whose block
sizes satisfy
$\gcd(b_1,\ldots,b_t)\mid q(q+1)A/2$.
\end{theorem}

\begin{proof}
Choose rank-two bundles $E_j$ with the displayed Chern classes, possible
by the four-dimensional classification of $\U(2)$ bundles~\cite[Theorem~1.4.20(a)]{GompfStipsicz}. Let $L$ be a
line bundle with $c_1(L)=-qa$. The Whitney formula gives
\[
 c_2\left(\bigoplus_jE_j\oplus L\right)
 =\left(\sum_jk_j+\binom q2 A-q^2A\right)[M]^*=0.
\]
Its first Chern class also vanishes, so its determinant-trivialized
$\SU(N)$ frame bundle is trivial over $M$. Choose determinant connections
on the $E_j$ with common curvature $-2\pi i\omega$ and use their inverse
tensor-product connection on $L$, of curvature $2\pi i q\omega$.
The resulting block connection has zero total trace. The real weights
make every $u_\sigma$
skew-Hermitian and traceless. Since $B=-2N\operatorname{tr}$ on
$\mathfrak{su}(N)$, the block trace is
$2\pi(\sum_j\lambda_j-q\lambda_0)\omega
=2\pi(2q+1)(\sum_j\lambda_j)\omega$, proving~\eqref{eq:charge-beta}.

Distinct weights force a gauge relating $u_{\id}$ and $u_\sigma$ to map
the rank-two eigenblocks according to $\sigma$. Such a gauge exists
exactly when the corresponding $E_j$ are isomorphic. Their common first
Chern class and the four-dimensional classification make this equivalent
to $k_j=k_{\sigma(j)}$ for every $j$. Choose literal copies of one model
bundle and one connection for all indices with the same $k_j$.
Permuting these copies gives a homomorphism
$\operatorname{Stab}_{S_q}(\mathbf k)\to\Gau(P)$; a permutation of
complex rank-two blocks has determinant $+1$, so these gauges preserve
the determinant trivialization and the chosen connection.

A $\Gamma$-fixed vector in~\eqref{eq:charge-lattice} has a constant
integer value $x_\ell$ on each $\Gamma$-orbit. Its defining equation is
$\sum_{\ell=1}^s d_\ell x_\ell=q(q+1)A/2$, which is solvable over
$\Z$ precisely under~\eqref{eq:charge-gcd}. For $\Gamma=S_q$ there is
one orbit of size $q$; its divisibility condition reduces to the final
parity statement.
An exact stabilizer of a charge vector permutes indices with equal
charges, and hence is a Young subgroup. Conversely, for a partition
with block sizes $b_1,\ldots,b_t$ satisfying the divisibility condition,
the equation $\sum_\ell b_\ell x_\ell=q(q+1)A/2$ has an integral
solution. If $t\geq2$, its integral solution set is infinite and is
not contained in any hyperplane $x_i=x_j$: its rational direction
space has a vector changing $x_i-x_j$, since every $b_\ell>0$.
Choose a solution with all $x_\ell$ distinct. Its stabilizer is
exactly the prescribed Young subgroup. The case $t=1$ is the full
$S_q$ condition already proved.
\end{proof}

\begin{remark}
The case $q=2$ recovers Theorem~\ref{thm:trivialP}. The charge lattice
records the part of the nonabelian reduction class invisible to both
determinant curvatures and the full paired two-form. Its Weyl fixed loci
are explicit affine sublattices, while~\eqref{eq:charge-gcd} detects when
a prescribed finite permutation symmetry cannot be realized on any
reduction in this family. Corollary~\ref{cor:elliptic} excludes a full
regular Weyl action when the paired curvature is symplectic and $G$ is
semisimple. In the present singular family, a full block-permutation
action is compatible with the same symplectic condition whenever the
parity criterion holds. The compensating line leaves a nonzero fixed
direction in the center of $H$.
\end{remark}

\section{Finite-order spectra of block mirrors}

Let $E$ and $F$ be Hermitian rank-$m$ bundles, with a trivialization of $\det(E\oplus F)$, and let $u=i(1_E,-1_F)$ on the associated $\SU(2m)$ bundle. A sign gauge exchanges the two eigenbundles. The spectral splitting of a finite-order bundle automorphism is the elementary finite-group case of equivariant bundle theory~\cite{Segal}. When a sign gauge exists, define its \emph{finite-order lift spectrum} to be
\[
 \Sigma(E,F)=\{\ord(J):J\text{ is a finite-order }\SU(2m)
 \text{ gauge with }JuJ^{-1}=-u\}.
\]

\begin{theorem}[Parity and eigenbundle spectrum]\label{thm:spectrum}
A sign gauge exists if and only if $E\simeq F$. If one exists, its least possible order is $2$ when $m$ is even and $4$ when $m$ is odd; for odd $m$ every finite order in $\Sigma(E,F)$ is divisible by four.

After identifying $E\simeq F$, the spectrum is exactly
\begin{equation}\label{eq:spectrum}
 \Sigma(E,F)=\{2\ord(U): U\in\operatorname{Aut}_{\U}(E)
 \text{ has finite order and }\det U=(-1)^m\}.
\end{equation}
Equivalently, $2s$ belongs to the spectrum precisely when $E$ admits a decomposition into eigenbundles $E=\bigoplus_{\zeta^s=1}E_\zeta$ for which the least common multiple of the orders of the occurring $\zeta$ is $s$ and
\begin{equation}\label{eq:cyclotomic}
 \prod_{\zeta^s=1}\zeta^{\operatorname{rank}E_\zeta}=(-1)^m.
\end{equation}
\end{theorem}

\begin{proof}
Every sign gauge is off diagonal, of the form $J=\left(\begin{smallmatrix}0&B\\ C&0\end{smallmatrix}\right)$ with unitary bundle isomorphisms $B:F\to E$ and $C:E\to F$. Thus it exists only if $E\simeq F$. Set $U=BC\in\operatorname{Aut}_{\U}(E)$. Direct calculation gives $\det J=(-1)^m\det U$, $J^2=\operatorname{diag}(U,CB)$, and $\ord(J)=2\ord(U)$ whenever $U$ has finite order. Conversely any unitary $U$ with $\det U=(-1)^m$, together with an isomorphism $E\simeq F$, defines such a gauge. This proves~\eqref{eq:spectrum}.

For even $m$ choose $U=I$. For odd $m$ choose $U=-I$. The latter has order two, giving a sign gauge of order four. If $m$ is odd and $U$ has finite order $s$, then $1=\det(U^s)=(-1)^s$, so $s$ is even. Finally, a finite-order unitary automorphism has constant eigenvalue multiplicities on connected $M$, and its spectral projections split $E$ into smooth subbundles. This proves the decomposition criterion; the converse uses scalar multiplication by $\zeta$ on each summand.
\end{proof}

\begin{corollary}[An order-six splitting detector]\label{cor:six}
For $m=2$ and $E\simeq F$, an order-six sign gauge exists if and only if $E$ splits as a direct sum of complex line bundles. In fact
\[
 \Sigma(E,F)=
 \begin{cases}
 \{2,4,6,8,\ldots\},&E\text{ splits into lines},\\
 \{2,4\},&E\text{ does not split into lines}.
 \end{cases}
\]
\end{corollary}

\begin{proof}
For rank two, $\det U=1$. If $U$ has order greater than two, it has two distinct eigenvalues: a scalar matrix of determinant one is only $\pm I$. Hence its eigenlines split $E$. Conversely, for $E=L_1\oplus L_2$, the automorphism $\operatorname{diag}(\zeta,\zeta^{-1})$ has determinant one and order $s$ for any primitive $s$th root $\zeta$. Formula~\eqref{eq:spectrum} gives every even order, including six. Without a splitting, only $U=\pm I$ have finite order, giving exactly $2$ and $4$.
\end{proof}

\begin{example}[Topology changes the spectrum]
Take $M=S^4$ and $E=F$ equal to the rank-two quaternionic Hopf bundle, with $c_2(E)=1$. Every complex line bundle on $S^4$ is trivial, so $E$ cannot split; its spectrum is $\{2,4\}$. If $E=F=\underline{\C}^2$, the spectrum contains every positive even integer. Both systems admit an involutive sign gauge and have identical pointwise $\SU(4)$ normalizer order. Their distinct spectra record global eigenbundle topology.
\end{example}

\begin{corollary}[Sphere-bundle order gaps]\label{cor:sphere-spectrum}
Let $m\geq2$, and let $E\to S^{2m}$ be a complex rank-$m$ bundle
with $c_m(E)\ne0$. Such bundles exist by Bott periodicity~\cite{Bott}.
For the block-sign field on $E\oplus E$, the complete finite-order
lift spectrum is
\[
 \Sigma(E,E)=
 \begin{cases}
 \{\,2d:d\mid m\,\},&m\text{ even},\\
 \{\,4d:d\mid m\,\},&m\text{ odd}.
 \end{cases}
\]
Thus rank three with $c_3(E)\ne0$ has spectrum $\{4,12\}$, while the
trivial rank-three bundle has every positive multiple of four.
\end{corollary}

\begin{proof}
If a finite-order unitary automorphism $U$ has distinct eigenvalues,
its spectral projections split $E$ into nonzero proper subbundles.
Every positive-degree cohomology group of $S^{2m}$ below degree $2m$
vanishes, so the Whitney formula would give $c_m(E)=0$. Hence every
such $U$ is scalar, say $U=\zeta I_E$. Formula~\eqref{eq:spectrum}
requires $\zeta^m=(-1)^m$ and gives $\ord(J)=2\ord(\zeta)$.
For even $m$, the possible orders of $\zeta$ are exactly the divisors
of $m$. For odd $m$, they are exactly twice the divisors of $m$.
If $E$ is trivial and $m=3$, the matrices
$\operatorname{diag}(\zeta,\zeta^{-1},-1)$, for primitive roots $\zeta$
of arbitrary even order, realize every positive multiple of four.
\end{proof}

\section{Fatness and the curvature-induced four-dimensional structure}

We now allow a circle reduction to move under its ambient connection.  The fatness condition makes this movement unavoidable and, on a closed four-manifold, imposes a topological obstruction to the Weyl sign operation on every circle reduction.

Let $X$ be a connected smooth four-manifold and let $P\to X$ be a
principal $\SO(3)$ bundle.  Put $E=P\times_{\SO(3)}\mathbb R^3$ and
identify $\mathfrak{so}(E)$ with $E$ using the cross product.  An
$\SO(3)$ connection $A$ has curvature $F_A\in\Omega^2(X;E)$.  Its
fiberwise transpose is the map
\begin{equation}\label{eq:curvature-map}
 \mathscr F_A:E\longrightarrow\Lambda^2T^*X,
 \qquad v\longmapsto\langle F_A,v\rangle.
\end{equation}
The connection is \emph{fully fat} when $\mathscr F_A(v)$ is
nondegenerate for every nonzero $v\in E_x$ and every $x\in X$.
Equivalently, $\mathscr F_A(v)\wedge\mathscr F_A(v)$ never vanishes.
This is Weinstein's curvature-pairing condition in the present rank-three setting~\cite{Weinstein}.

\begin{lemma}[Fatness excludes parallel circle reductions]\label{lem:no-parallel-fat}
If an $\SO(3)$ connection $A$ is fully fat, then it has no nonzero parallel adjoint field.  Every maximal-circle reduction of its underlying bundle moves under $A$.
\end{lemma}

\begin{proof}
If $\nabla_Au=0$, then $[F_A,u]=\nabla_A^2u=0$.  The centralizer of a nonzero vector in $\mathfrak{so}(3)$ is its own span, so $F_A$ takes values in $\R u$.  Pairing with any nonzero $y\perp u$ gives the zero two-form, contrary to full fatness.  An $A$-preserved circle reduction would supply such a parallel $u$.
\end{proof}

\begin{theorem}\label{thm:general}
Suppose $A$ is fully fat.  Its curvature determines an orientation
$o_A$ and a conformal structure $[g_A]$ on $X$ for which
\begin{equation}\label{eq:selfdual-identification}
 E\simeq\Lambda^+_{[g_A],o_A}T^*X
\end{equation}
as oriented rank-three bundles.  The following conditions are equivalent:
\begin{enumerate}
\item $P$ reduces to a maximal $\SO(2)\subset\SO(3)$;
\item $E$ has a nowhere-zero section;
\item $X$ has an almost complex structure inducing $o_A$.
\end{enumerate}
The first condition concerns the underlying bundle; it imposes no
holonomy reduction on $A$.
\end{theorem}

\begin{proof}
Fatness implies that $\mathscr F_A$ is injective.  In a local
orthonormal frame $(e_1,e_2,e_3)$ of $E$, the four-form
\[
 \tau_A=\sum_{i=1}^3
 \mathscr F_A(e_i)\wedge\mathscr F_A(e_i)
\]
is independent of the frame and hence globally defined.  At each
point its three summands are nonzero and have the same sign relative
to either local orientation: the unit sphere of $E_x$ is connected,
and $v\mapsto\mathscr F_A(v)^2$ never vanishes on it.
Consequently $\tau_A$ is nowhere zero and orients $X$; denote this
orientation by $o_A$.  Wedge product on $\Lambda^2T_x^*X$ now gives a
quadratic form of signature $(3,3)$, and
$\mathscr F_A(E_x)$ is a positive three-plane.  The correspondence
between positive three-planes
in $\Lambda^2T_x^*X$ and oriented conformal inner products on $T_xX$
then yields $[g_A]$, smoothly in $x$, with
$\mathscr F_A(E)=\Lambda^+_{[g_A],o_A}T^*X$; see also
\cite[Section~2.2]{FP}.  Polar decomposition of the fiberwise
isomorphism $\mathscr F_A$ makes~\eqref{eq:selfdual-identification}
an isomorphism of Euclidean bundles.  If its orientation is reversed,
compose with $-\mathrm{id}$ on the rank-three fibers.

An $\SO(2)$ reduction is precisely a unit section of $E$, since
$\SO(2)$ is the stabilizer of a unit vector in $\SO(3)$.  By
\eqref{eq:selfdual-identification}, this is a unit self-dual two-form
for $g_A$.  In oriented dimension four, such a form is the fundamental
form of an orientation-compatible orthogonal almost complex structure,
after fixing its conventional length.  Conversely, any almost complex
structure compatible with $o_A$ can be made orthogonal for $g_A$.
Indeed, choose a Hermitian metric $h$ for it, interpolate through positive
metrics from $h$ to $g_A$, and apply the polar part of the skew
endomorphism defined by its fundamental form.  The resulting
$g_A$-orthogonal almost complex structure supplies the self-dual
unit section.
\end{proof}

\begin{corollary}[Characteristic-class reduction test]
Let $X$ be closed and let $o_A$ be the orientation supplied by a fully
fat $\SO(3)$ connection.  Its underlying bundle reduces to $\SO(2)$
if and only if there is a class $c\in H^2(X;\Z)$ such that
\[
 c\bmod 2=w_2(TX),\qquad
 \langle c^2,[X]_{o_A}\rangle=2\chi(X)+3\sigma(X,o_A).
\]
For a reduction, $c$ can be taken as the Euler class of its oriented
orthogonal two-plane bundle, up to sign.
\end{corollary}

\begin{proof}
The first two conditions are exactly the Hirzebruch--Hopf--Wu
criterion for an $o_A$-compatible almost complex structure
\cite{HirzebruchHopf}.  Apply Theorem~\ref{thm:general}.  For the final
assertion, the induced complex tangent bundle has first Chern class
equal, up to sign, to the Euler class of the plane bundle obtained by
reducing $\Lambda^+$ to $\SO(2)$.
\end{proof}

\begin{theorem}[Fatness breaks circle-reduction sign symmetry]\label{thm:fat-mirror}
Let $X$ be a closed connected four-manifold and let $A$ be a fully fat connection on a principal $\SO(3)$ bundle $P$.  Orient $X$ by the positive curvature three-plane of Theorem~\ref{thm:general}.  Then
\[
 \langle p_1(\ad P),[X]_{o_A}\rangle>0.
\]
If $Q\subset P$ is any maximal-circle reduction with unit adjoint field $u$ and Euler class $c=e(u^\perp)$, then $c^2=p_1(\ad P)$, so $c$ is not torsion.  No gauge transformation sends $u$ to $-u$.  The reduction is necessarily moving under $A$.
\end{theorem}

\begin{proof}
In a local oriented orthonormal frame $(e_1,e_2,e_3)$ of $\ad P$, write $F_A=\sum_i\Omega_i e_i$ using the cross-product identification of $\mathfrak{so}(3)$ with $\R^3$.  Chern--Weil theory gives
\[
 p_1(\ad P)=\left[\frac{1}{4\pi^2}\sum_i\Omega_i\wedge\Omega_i\right].
\]
The four-form under the brackets is positive for $o_A$ at every point by full fatness, proving the first assertion.  The reduction splits $\ad P=\R u\oplus u^\perp$, so the Whitney formula gives $p_1(\ad P)=e(u^\perp)^2=c^2$.  If a gauge sent $u$ to $-u$, its restriction to the oriented two-plane $u^\perp$ would reverse that plane's orientation while covering the identity of $X$.  Naturality of Euler class would give $c=-c$, contradicting the positive square of $c$.  The final assertion follows from Lemma~\ref{lem:no-parallel-fat}.
\end{proof}

\begin{remark}
The conformal metric in Theorem~\ref{thm:general} is the one encoded
by the definite curvature three-plane.  Its construction and the
equivalence between definite $\SO(3)$ connections and fatness in
dimension four are part of the definite-connection framework of
Fine and Panov~\cite[Sections~2.1--2.2]{FP}.  The theorem records the
resulting reduction criterion with the bundle and orientation fixed.
\end{remark}

\section{Hyperbolic classification and a reduction counterexample}

\begin{theorem}\label{thm:hyperbolic}
Let $(X^4,g)$ be a closed oriented real-hyperbolic manifold of sectional
curvature $-1$, and set
\[
 P_X=\Fr^+(\Lambda^+_gT^*X).
\]
The connection $A_g$ induced by the Levi--Civita connection is fully
fat.  The underlying $\SO(3)$ bundle $P_X$ reduces to $\SO(2)$ if
and only if
\begin{equation}\label{eq:euler-mod-four}
 \chi(X)\equiv0\pmod4.
\end{equation}
There are closed manifolds satisfying~\eqref{eq:euler-mod-four}.
Consequently a fully fat principal $\SO(3)$ bundle can reduce to a
proper connected subgroup.  The same phenomenon occurs for an
$\SU(2)$ bundle reducing to $\mathrm U(1)$.  For either example and
every nonzero $y$ orthogonal to the Lie algebra of the indicated
circle subgroup, the fat connection is $y$-fat.
\end{theorem}

\begin{proof}
The curvature operator of a metric of sectional curvature $-1$ is
$-\mathrm{id}$ on $\Lambda^2T^*X$.  Under
$\mathfrak{so}(4)=\mathfrak{so}(3)_+\oplus\mathfrak{so}(3)_-$, the
curvature of the induced connection on $\Lambda^+$ is the positive
summand of this operator.  Equivalently, the Fine--Panov formula
$F_{A_g}=(W^++s/12)\oplus\operatorname{Ric}_0$ gives
$W^+=0$, $s=-12$, $\operatorname{Ric}_0=0$, and hence
\begin{equation}\label{eq:hyperbolic-curvature}
 \mathscr F_{A_g}=c\,\mathrm{id}_{\Lambda^+}
 \quad\text{for a fixed }c\ne0
\end{equation}
under the standard identifications.  The value of $c$ depends only on
the normalization of $\mathfrak{so}(3)\simeq\Lambda^+$; in
\cite[Section~3.1]{FP} it is $-1$.  A nonzero self-dual two-form
$\omega$ satisfies $\omega^2=|\omega|^2\vol_g>0$.
Equation~\eqref{eq:hyperbolic-curvature} therefore proves full
fatness for every nonzero Lie-algebra direction.

An $\SO(2)$ reduction of $P_X$ is a section of $S(\Lambda^+)$,
equivalently an almost complex structure inducing the chosen
orientation.  Kotschick's criterion~\cite[Proposition~11]{Kotschick}
gives precisely~\eqref{eq:euler-mod-four}.  Conder and Maclachlan
constructed a closed orientable hyperbolic four-manifold with
$\chi=16$~\cite{ConderMaclachlan}; Kotschick observes that it is
almost complex~\cite[Example~15]{Kotschick}.  This proves the
$\SO(3)$ assertion and contradicts the unrestricted reduction
conjecture of~\cite{ZillerFatness} as recorded in~\cite[Section~5]{FZ}.  Its $\SO(2)$ subgroup
is nonnormal, so the normal-subgroup theorem of~\cite{FZ} remains
applicable in its stated range.  Since $A_g$ is fully fat, every
nonzero vector in the orthogonal complement of this circle Lie
algebra is fat.  Thus the example also contradicts the stronger
$y$-fat conjecture of~\cite[Section~5]{FZ}.

For the $\SU(2)$ assertion, take a finite spin cover $Y$ of the
$\chi=16$ example.  Such a cover exists by the virtual-spin theorem
for real-hyperbolic manifolds, originating in Sullivan's application
of Deligne--Sullivan's work~\cite{DeligneSullivan,Sullivan,LongReid}.
Since $\chi(Y)$ remains divisible by four, $Y$ is almost complex.
A spin structure lifts $P_Y=\Fr^+(\Lambda^+T^*Y)$ to a principal
$\SU(2)$ bundle $\widetilde P_Y$: the $\SU(2)_+$ factor of
$\operatorname{Spin}(4)=\SU(2)_+\times\SU(2)_-$ maps onto the
$\SO(3)$ action on $\Lambda^+$.  The connection $A_g$ lifts and remains fully fat,
because $\mathfrak{su}(2)\to\mathfrak{so}(3)$ is an isomorphism.
The inverse image in $\widetilde P_Y$ of the $\SO(2)$ reduction of
$P_Y$ is a principal $\mathrm U(1)$ reduction.
\end{proof}

\begin{corollary}[Reduction changes under finite covers]
The connection $A_g$ remains fully fat on every finite cover of $X$.
The Davis hyperbolic four-manifold has $\chi=26$ and its bundle
$P_X$ has no $\SO(2)$ reduction, whereas every connected even-degree
cover has a reducible fully fat bundle of the same type.
\end{corollary}

\begin{proof}
Fatness is a pointwise condition and survives pullback by a covering.
The Euler characteristic multiplies by the covering degree.  Apply
Theorem~\ref{thm:hyperbolic} to the Davis manifold and its covers;
the value $\chi=26$ and existence of even-degree covers are recorded
in~\cite[Example~14 and Corollary~12]{Kotschick}.
\end{proof}

\section{Moving reductions and the closed curvature they detect}

Let $G$ be $\SO(3)$ or $\SU(2)$, let $H\simeq S^1$ be a maximal
torus, and let $Q\subset P$ be an $H$ reduction over a connected
even-dimensional manifold $M^{2m}$.  Fix a nonzero
$\mu\in\mathfrak h$.  It defines an adjoint field $u$ by
$u(qg)=\operatorname{Ad}_{g^{-1}}\mu$.  For an arbitrary
$G$ connection $A$ on $P$, let
\begin{equation}\label{eq:decomposition}
 A|_Q=a+\phi,
 \qquad a=\operatorname{pr}_{\mathfrak h}(A|_Q),
 \quad\phi\in\Omega^1(Q;\mathfrak h^\perp).
\end{equation}
The orthogonal projection $a$ is a connection on $Q$.  Set
$\alpha=\langle u,A\rangle\in\Omega^1(P)$ and let
$\beta_a\in\Omega^2(M)$ be the closed form whose pullback to $Q$ is
$\langle\mu,F_a\rangle$.

\begin{proposition}[Circle scalarization]\label{prop:moving}
With the preceding conventions, $\alpha$ depends on $A$ only through
the projected circle connection $a$.  The flow generated by
$u(p)^\#$ is periodic, and its orbit space is canonically
$M\times(H\backslash G)\simeq M\times S^2$.  Under this
identification, $\dd\alpha$ is the pullback of
\begin{equation}\label{eq:product-form}
 \beta_a+\omega_\mu,
\end{equation}
where $\omega_\mu$ is the invariant nondegenerate area form on
$H\backslash G$ determined by $\mu$.  Consequently $\alpha$ is a
contact form exactly when $\beta_a$ is symplectic on $M$.
\end{proposition}

\begin{proof}
Use the quotient presentation $P=Q\times_HG$ and pull forms back
along $(q,g)\mapsto qg$.  Equivariance of $u$ and $A$ gives
\begin{equation}\label{eq:pullback-alpha}
 \alpha=\langle\mu,A|_Q\rangle
       +\langle\mu,\dd g\,g^{-1}\rangle
       =\langle\mu,a\rangle
       +\langle\mu,\dd g\,g^{-1}\rangle.
\end{equation}
The vector field $u(p)^\#$ is represented by left multiplication
$g\mapsto\exp(t\mu)g$.  Formula~\eqref{eq:pullback-alpha} shows
that $\alpha(u^\#)=|\mu|^2$ and $\iota_{u^\#}\dd\alpha=0$.
Its orbits are the $H$-orbits on the $G$ coordinate, and the map
$[q,g]\mapsto(\pi(q),Hg)$ identifies the orbit space with the stated
product.  Differentiating~\eqref{eq:pullback-alpha} gives the sum
of the basic curvature $\beta_a$ and the area form induced by the
right Maurer--Cartan form; this is~\eqref{eq:product-form}.
The restriction of $\omega_\mu$ to $S^2$ is nondegenerate.  Since
$\dim M=2m$, the top exterior power of~\eqref{eq:product-form}
is $(m+1)\beta_a^m\wedge\omega_\mu$.  This is nowhere zero
precisely when $\beta_a$ is symplectic.
\end{proof}

\begin{corollary}[Fixed-reduction contact criterion]
Fix a maximal-circle reduction $Q\subset P$ and its adjoint field
$u$.  The following are equivalent:
\begin{enumerate}
\item some connection $A$ on $P$ makes $\langle u,A\rangle$ a
contact form;
\item the real cohomology class of the circle curvature
$\beta_a$ contains a symplectic form on $M$.
\end{enumerate}
The class in the second condition depends only on $Q$ and the
normalization of $\mu$.
\end{corollary}

\begin{proof}
Proposition~\ref{prop:moving} gives the forward implication.
Conversely, let $\omega$ be a symplectic representative of
$[\beta_a]$.  The difference $\omega-\beta_a$ is exact.  Adding an
$\mathfrak h$-valued one-form pulled back from $M$ to the circle
connection $a$ changes its paired curvature to $\omega$.  Extend
this new circle connection to a $G$ connection on $P=Q\times_HG$.
Proposition~\ref{prop:moving} then makes $\langle u,A\rangle$
contact.
\end{proof}

The paired curvature $\langle u,F_A\rangle$ is generally not closed
when the reduction moves under $A$.  The circle curvature $\beta_a$
contains the missing derivative term.  In the $\SO(3)$ case, identify
$E=\ad P$ as an oriented Euclidean rank-three bundle and take $u$
to be the unit rotation generator, whose one-parameter subgroup has
period $2\pi$.  Thus $\beta_a$ below is the curvature paired with this
normalized generator; replacing $u$ by a scalar multiple scales
$\beta_a$ by the same multiple.  Write $\nabla_A$ for its induced covariant
derivative, and orient $u^\perp$ by the orientation of $E$ and $u$.

\begin{proposition}[Curvature with a moving reduction]\label{prop:defect}
For tangent vectors $V,W\in T_xM$,
\begin{equation}\label{eq:defect-curvature}
 \beta_a(V,W)=\langle u,F_A(V,W)\rangle
 -\langle u,(\nabla_Vu)\times(\nabla_Wu)\rangle,
\end{equation}
after identifying the circle generator with rotation about $u$.
In particular $[\beta_a/(2\pi)]=e(u^\perp)$ up to the chosen
generator sign, and on a closed oriented four-manifold
\begin{equation}\label{eq:pontryagin-square}
 \int_M\left(\frac{\beta_a}{2\pi}\right)^2
 =\langle p_1(E),[M]\rangle.
\end{equation}
\end{proposition}

\begin{proof}
Choose a local oriented orthonormal frame $(e_1,e_2,u)$ of $E$.
Write $\nabla u=D_1e_1+D_2e_2$ and
$\nabla e_1=\gamma e_2-D_1u$,
$\nabla e_2=-\gamma e_1-D_2u$.
The projected circle connection has scalar curvature $\dd\gamma$.
The component of the full curvature in the $u$ direction is
$\dd\gamma+D_1\wedge D_2$, while
$\langle u,(\nabla_Vu)\times(\nabla_Wu)\rangle
 =(D_1\wedge D_2)(V,W)$.
This proves~\eqref{eq:defect-curvature}.  Chern--Weil theory gives
$[\dd\gamma/(2\pi)]=\pm e(u^\perp)$, with the sign fixed by the
chosen circle generator.
Since $E=\mathbb Ru\oplus u^\perp$, the Whitney formula gives
$p_1(E)=e(u^\perp)^2$, proving~\eqref{eq:pontryagin-square}.
\end{proof}

\begin{corollary}[A positive cup square without contact]\label{cor:nosymplectic}
There is a closed oriented four-manifold $M$ and a principal
$\SO(3)$ bundle $P\to M$ with a fully fat connection and an
$\SO(2)$ reduction such that, for every connection $A'$ on $P$
and every nonzero constant-length adjoint field $u'$, the form
$\langle u',A'\rangle$ is not contact.  Every circle reduction of
this $P$ nevertheless has Euler class $e$ with
$\langle e^2,[M]\rangle=2\chi(M)>0$.
\end{corollary}

\begin{proof}
Kotschick~\cite[Example~16]{Kotschick} describes the closed
real-hyperbolic double cover constructed by Agol and Lin~\cite{AgolLin}:
its Euler characteristic is divisible by four and its
Seiberg--Witten invariants vanish, excluding symplectic forms.
Theorem~\ref{thm:hyperbolic} gives a fat connection and a circle
reduction on $P=\Fr^+(\Lambda^+T^*M)$.  If some
$\langle u',A'\rangle$ were contact, Proposition~\ref{prop:moving}
would give a symplectic form on $M$, a contradiction.  Finally,
$p_1(\Lambda^+)=p_1(TM)+2e(TM)$; the hyperbolic signature is zero,
so~\eqref{eq:pontryagin-square} evaluates to $2\chi(M)>0$.
\end{proof}

\begin{proposition}[Covariant cost of a hyperbolic reduction]
For any unit section $u\in\Gamma(\Lambda^+T^*X)$ on a closed
real-hyperbolic four-manifold,
\begin{equation}\label{eq:bochner-cost}
 \int_X|\nabla_{A_g}u|^2\,\vol_g
 =4\,\vol_g(X)+2\int_X|\dd u|^2\,\vol_g
 \geq 4\,\vol_g(X).
\end{equation}
In particular no circle reduction is $A_g$-parallel.
\end{proposition}

\begin{proof}
The self-dual Weitzenb\"ock formula is
$\Delta u=\nabla^*\nabla u-2W^+(u)+(s/3)u$.
Here $W^+=0$ and $s=-12$, so
$\Delta u=\nabla^*\nabla u-4u$.
For a self-dual two-form, $|\dd^*u|=|\dd u|$ pointwise.
Taking the $L^2$ inner product with $u$ and using $|u|=1$ gives
\eqref{eq:bochner-cost}.
\end{proof}

\appendix

\section{Automorphism twists and holonomy}\label{sec:twist}

Outer automorphism actions on moduli of principal bundles have also been studied in the holomorphic setting~\cite{AntonSancho}.  Here we keep the underlying smooth bundle fixed only when its twist is isomorphic to it.

For $\varphi\in\operatorname{Aut}(G)$ define $P^\varphi$ to have the same total manifold and projection as $P$, with right action
\[
 p\mathbin{\star_\varphi}g=p\,\varphi^{-1}(g).
\]
The identity map $j_\varphi:P\to P^\varphi$ is $\varphi$-semilinear.  On $P^\varphi$ set
\begin{equation}\label{eq:twistdata}
 A^\varphi=\dd\varphi\circ A,
 \qquad u^\varphi=\dd\varphi\circ u.
\end{equation}
The formula for the coadjoint field is
\[
 \lambda^\varphi=(\dd\varphi^{-1})^*\lambda,
\]
since $\langle\lambda^\varphi,\dd\varphi(a)\rangle=\langle\lambda,a\rangle$.

Let $\operatorname{Aut}_{\mathrm{tw}}(P)$ be the group of pairs $(\varphi,\Phi)$, where $\varphi\in\operatorname{Aut}(G)$ and $\Phi:P\to P$ covers $\id_M$ and satisfies $\Phi(pg)=\Phi(p)\varphi(g)$.  Composition is componentwise.  The stabilizer of the bundle class under twisting is denoted $\operatorname{Stab}_{\operatorname{Aut}(G)}([P])$.

\begin{theorem}[Twist naturality and self-lifts]\label{thm:twist}
The data in~\eqref{eq:twistdata} form a parallel compatible pair on $P^\varphi$ and satisfy
\begin{equation}\label{eq:twistnaturality}
 j_\varphi^*\alpha^\varphi=\alpha,
 \qquad j_\varphi^*B(u^\varphi,F_{A^\varphi})=B(u,F_A),
 \qquad (j_\varphi)_*\mathcal D=\mathcal D^\varphi.
\end{equation}
A $\varphi$-semilinear self-map $\Phi:P\to P$ over $\id_M$ exists if and only if $P^\varphi\cong P$ as principal $G$-bundles.  If one such map exists, all such maps form a torsor under $\Gau(P)$.  In particular, there is an exact sequence
\begin{equation}\label{eq:exacttwist}
1\longrightarrow\Gau(P)\longrightarrow
\operatorname{Aut}_{\mathrm{tw}}(P)\longrightarrow
\operatorname{Stab}_{\operatorname{Aut}(G)}([P])\longrightarrow1,
\end{equation}
where the last group consists of $\varphi$ with $P^\varphi\cong P$.
\end{theorem}

\begin{proof}
The new fundamental vector for $a\in\g$ is the old one for $\dd\varphi^{-1}a$, so $A^\varphi$ reproduces $a$.  Its right equivariance follows from the intertwining of $\Ad$ by $\varphi$.  Also
\[
 \dd_{A^\varphi}u^\varphi
 =\dd\varphi(\dd_Au)=0,
 \qquad F_{A^\varphi}=\dd\varphi(F_A).
\]
The Killing form is invariant under $\dd\varphi$, proving~\eqref{eq:twistnaturality}.  A semilinear $\Phi:P\to P$ is exactly an ordinary principal-bundle isomorphism $P^\varphi\to P$ after composing with $j_\varphi$ in the appropriate direction.  Two such lifts differ by an ordinary gauge transformation.  Taking the kernel and image of the projection $(\varphi,\Phi)\mapsto\varphi$ gives~\eqref{eq:exacttwist}.
\end{proof}

The subgroup of inner automorphisms always lies in the stabilizer in~\eqref{eq:exacttwist}.  For $\varphi=\Ad_h$, the map $p\mapsto p h^{-1}$ is $\varphi$-semilinear.  Section~\ref{sec:examples} gives a parallel pair for which an outer automorphism has no self-lift.

\begin{theorem}[Holonomy criterion]\label{thm:holonomy}
Fix $p_0\in P$, put $K=\Hol_{p_0}(A)$ and $\mu=u(p_0)$.  A $\varphi$-semilinear map $\Phi:P\to P$ over $\id_M$ satisfying
\begin{equation}\label{eq:preserveA}
 \Phi^*A=\dd\varphi\circ A
\end{equation}
and $\Phi(p_0)=p_0h$ exists if and only if
\begin{equation}\label{eq:holcriterion}
 kh=h\varphi(k)\qquad\text{for every }k\in K.
\end{equation}
It satisfies $u\circ\Phi=\dd\varphi\circ u$ exactly when $\Ad_{h^{-1}}\mu=\dd\varphi(\mu)$.  In particular, an ordinary gauge transformation preserving $A$ and carrying $u$ to $-u$ exists exactly when
\begin{equation}\label{eq:fixedAcriterion}
 \text{some }h\in Z_G(K)\text{ satisfies }\Ad_{h^{-1}}\mu=-\mu.
\end{equation}
\end{theorem}

\begin{proof}
Condition~\eqref{eq:preserveA} sends $A$-horizontal paths to $A$-horizontal paths.  Lift a loop at $\pi(p_0)$ from $p_0$ to $p_0k$.  Starting instead at $p_0h$, its horizontal lift ends at $p_0kh$; semilinearity says that applying $\Phi$ to the original endpoint gives $p_0h\varphi(k)$.  This proves necessity of~\eqref{eq:holcriterion}.

Conversely, for a path $\gamma$ from $\pi(p_0)$ to $x$, let $p_\gamma$ be its horizontal endpoint starting at $p_0$.  Define
\[
 \Phi(p_\gamma g)=p_\gamma h\varphi(g).
\]
Two paths to $x$ differ by an element of $K$; equation~\eqref{eq:holcriterion} makes the definition independent of the path.  It is smooth, semilinear, and satisfies~\eqref{eq:preserveA}.  Parallelism of $u$ reduces the field relation to its value at $p_0$, giving the displayed condition.  Taking $\varphi=\id_G$ yields~\eqref{eq:fixedAcriterion}.
\end{proof}

\section{Atiyah cohomology and the first-jet Spencer operator}\label{sec:cohomology}

The Atiyah algebroid of $P$ is $\mathcal A_P=TP/G\to M$.  Its sections are the $G$-invariant vector fields on $P$; the bracket comes from vector fields and the anchor from $\dd\pi$.  Its Chevalley--Eilenberg forms are precisely the invariant differential forms on $P$:
\begin{equation}\label{eq:AtiyahDGA}
 \bigl(\Omega^\bullet(\mathcal A_P),\dd_{\mathcal A_P}\bigr)
 \cong\bigl(\Omega^\bullet(P)^G,\dd\bigr)
\end{equation}
as differential graded algebras.  This is a genuine cochain complex associated with the principal bundle, with its full topology retained.

\begin{proposition}[Cohomology and automorphism transport]\label{prop:Atiyahcohomology}
For compact connected $G$, inclusion of invariant forms induces an isomorphism of graded rings
\begin{equation}\label{eq:AtiyahH}
 H^\bullet(\mathcal A_P)\cong H^\bullet_{\mathrm{dR}}(P).
\end{equation}
For every $\varphi\in\operatorname{Aut}(G)$, the canonical semilinear map $j_\varphi:P\to P^\varphi$ induces an isomorphism of the complexes~\eqref{eq:AtiyahDGA}.  An actual $\varphi$-semilinear self-lift of $P$ acts on~\eqref{eq:AtiyahH} as a graded-ring automorphism.
\end{proposition}

\begin{proof}
Identification~\eqref{eq:AtiyahDGA} follows by evaluating an algebroid form on invariant vector fields and then evaluating at each point of $P$.  Average forms over $G$ with normalized Haar measure.  Averaging commutes with $\dd$ and is a left inverse of inclusion.  Every right translation $R_g$ is isotopic to the identity because $G$ is connected, so it acts trivially on de Rham cohomology.  Averaging consequently acts as the identity on cohomology, proving~\eqref{eq:AtiyahH}.  A semilinear map takes invariant forms to invariant forms and commutes with exterior differentiation and wedge product; $j_\varphi$ is a diffeomorphism, hence induces a complex isomorphism.
\end{proof}

For comparison with the infinitesimal terminology, let $E\to M$ be any vector bundle.  Its first-jet sequence is
\begin{equation}\label{eq:jetseq}
0\longrightarrow T^*M\otimes E\xrightarrow{\ i\ }J^1E
 \xrightarrow{\operatorname{pr}}E\longrightarrow0.
\end{equation}
We choose the injection convention $i(\dd f\otimes s)=f j^1s-j^1(fs)$.  The canonical Spencer operator is the first-order map
\begin{equation}\label{eq:Spencer}
 \mathsf D_E:\Gamma(J^1E)\longrightarrow\Omega^1(M;E),
 \quad \mathsf D_E(j^1s)=0,
 \quad \mathsf D_E(f\xi)=f\mathsf D_E\xi+\dd f\otimes\operatorname{pr}(\xi).
\end{equation}
In this convention $\mathsf D_E\circ i=\id_{T^*M\otimes E}$.  These formulas define the standard jet Spencer operator~\cite{Spencer,CSS}.

\begin{proposition}[Spencer naturality]\label{prop:Spencernatural}
If $T:E\to E'$ is a vector-bundle isomorphism covering $\id_M$, then its first-jet map satisfies
\begin{equation}\label{eq:Spencernatural}
 \mathsf D_{E'}\circ J^1T=(\id_{T^*M}\otimes T)\circ\mathsf D_E.
\end{equation}
In particular~\eqref{eq:Spencernatural} applies to the adjoint-bundle isomorphism induced by an existing semilinear bundle map, and to the sign map on the parallel line subbundle $\R u\subset\operatorname{ad}P$.
\end{proposition}

\begin{proof}
Both sides vanish on $j^1s$.  The kernel injection in~\eqref{eq:jetseq} is natural, so $J^1T\circ i=i\circ(\id_{T^*M}\otimes T)$; both sides of~\eqref{eq:Spencernatural} agree on $i(T^*M\otimes E)$ because $\mathsf D\circ i=\id$.  Every jet section $\xi$ decomposes as $j^1(\operatorname{pr}\xi)+i(\eta)$ for a unique one-form $\eta$ with values in $E$.  The two maps therefore agree on every section.
\end{proof}

The jet Spencer operator~\eqref{eq:Spencer} records infinitesimal compatibility of first jets, while~\eqref{eq:AtiyahDGA} carries the graded-ring cohomology used in this paper.

\section{Examples and separated obstructions}\label{sec:examples}

\begin{example}[A contact constraint on a trivial $\mathrm{SU}(2)$ bundle]\label{ex:T2}
Let $M=T^2$, and let $L\to T^2$ have $c_1(L)=k[T^2]$, where $k\in\Z\setminus\{0\}$.  Put $E=L\oplus L^{-1}$, with determinant trivialization, and let $P$ be its $\mathrm{SU}(2)$ frame bundle.  Every $\mathrm{SU}(2)$ bundle over a two-dimensional CW complex is trivial, so $P\cong T^2\times S^3$ as a principal bundle.  Choose a unitary connection on $L$ with nonzero constant curvature and the dual connection on $L^{-1}$; together they give a connection $A$ on $P$.  The diagonal field $u$ with eigenvalues $i,-i$ is $A$-parallel.  Its paired curvature $\beta$ is a nonzero multiple of the area form on $T^2$.  Proposition~\ref{prop:levi} shows that $\alpha=B(u,A)$ is a contact form on the five-manifold $P$.

The two fields $u$ and $-u$ are conjugate in every fiber, and their paired forms have the same kernel.  Yet $2c_1(L)=2k[T^2]\ne0$, so Corollary~\ref{cor:su2criterion} excludes every gauge transformation carrying $u$ to $-u$.  The obstruction is carried by the reduction $Q$ inside a trivial ambient bundle, as anticipated by the diagonalization phenomena of~\cite{BlauThompson}.
\end{example}

The contact distribution retains a Chern-class obstruction even for diffeomorphisms that do not preserve the principal action.

\begin{proposition}[Contact-plane Chern obstruction]\label{prop:contactChern}
Let $(P,A,u)$ be a contact parallel pair for $G=\mathrm{SU}(2)$ over a connected base $M^{2m}$, and let $J_\beta$ be a $\beta$-compatible almost complex structure on $TM$.  With $e=c_1(L)$, the first Chern class of the contact plane is
\begin{equation}\label{eq:contactc1}
 c_1(\mathcal D,\dd\alpha|_{\mathcal D})
 =\pi^*\bigl(c_1(TM,J_\beta)+\varepsilon\,2e\bigr),
 \qquad \varepsilon\in\{1,-1\}.
\end{equation}
Any diffeomorphism $F:P\to P$ covering $\id_M$ that preserves $\mathcal D$ and reverses its coorientation requires
\begin{equation}\label{eq:chernnecessary}
 2\bigl(c_1(TM,J_\beta)+\varepsilon\,2e\bigr)=0
 \quad\text{in }H^2(M;\Z).
\end{equation}
For Example~\ref{ex:T2}, $c_1(\mathcal D)=\varepsilon\,2k\pi^*[T^2]\ne0$, and no such diffeomorphism exists.
\end{proposition}

\begin{proof}
The splitting~\eqref{eq:hyperplanesplit} is symplectic for $\dd\alpha$: its horizontal part is $\pi^*TM$ with form $\pi^*\beta$, and its vertical part is the pullback of the real rank-two bundle $Q\times_{S^1}\mu^\perp$ with the Kirillov--Kostant--Souriau form.  The horizontal block has first Chern class $\pi^*c_1(TM,J_\beta)$.  Using $E=L\oplus L^{-1}$, the vertical complex line is $\operatorname{Hom}(L^{-1},L)=L^2$, or its conjugate if the compatible complex orientation is reversed.  Its first Chern class is $\varepsilon 2\pi^*e$, proving~\eqref{eq:contactc1}.

The Gysin sequence of the $S^3$-bundle $P\to M$ gives an isomorphism $\pi^*:H^2(M;\Z)\to H^2(P;\Z)$.  A diffeomorphism covering $\id_M$ acts trivially on this group.  If it preserved $\mathcal D$ and reversed coorientation, it would take $\dd\alpha|_{\mathcal D}$ to a negative positive-function multiple of itself.  Conjugating any compatible complex structure changes the first Chern class by a sign, so $F^*c_1(\mathcal D)=-c_1(\mathcal D)$.  Injectivity of $\pi^*$ gives~\eqref{eq:chernnecessary}.  On $T^2$, $TT^2$ is trivial and $e=k[T^2]$ with $k\ne0$, so the necessary condition fails.
\end{proof}

\begin{example}[A flat torsion obstruction]
Let $M=L(3,1)$.  A nontrivial character $\pi_1(M)=\Z/3\to S^1$ defines a flat circle bundle $Q$ with $c_1(Q)$ generating $H^2(M;\Z)=\Z/3$.  Extend $Q$ to a principal $\mathrm{SU}(2)$ bundle $P$ and extend its flat connection to $A$.  A nonzero diagonal $u$ is parallel and $F_A=0$.  Principal $\mathrm{SU}(2)$ bundles over three-dimensional CW complexes are trivial, so $P$ is trivial.  Nevertheless $2c_1(Q)\ne0$ in $\Z/3$, and Corollary~\ref{cor:su2criterion} rules out a gauge sign realization.  This obstruction is torsion and all Chern--Weil curvature forms vanish.
\end{example}

\begin{example}[Field symmetry without connection symmetry]
On a trivial $\mathrm{SU}(2)$ bundle over $S^1$, choose a flat diagonal connection with holonomy generated by an element of infinite order in the maximal torus.  The parallel diagonal field $u$ has a trivial eigenline, so Corollary~\ref{cor:su2criterion} supplies a gauge transformation carrying $u$ to $-u$.  The holonomy subgroup is dense in the torus, whose centralizer is that torus.  It fixes $u$, so~\eqref{eq:fixedAcriterion} excludes every such gauge transformation that also preserves $A$.
\end{example}

\begin{example}[Contactness forces a nonzero Serre differential]\label{ex:T4}
Write $T^4=T^2_1\times T^2_2$, and let $a,b$ be the positive integral area classes of the factors.  Choose a line bundle $L$ with $c_1(L)=a+b$ and a connection with curvature a nonzero multiple of the constant form representing $a+b$.  The construction of Example~\ref{ex:T2} gives a parallel pair on the $\mathrm{SU}(2)$ frame bundle of $E=L\oplus L^{-1}$.  Its paired curvature is symplectic, hence its constraint form is contact on the seven-manifold $P$.  Here
\[
 c_2(E)=-(a+b)^2=-2ab\ne0.
\]
The fiber class in $H^3(S^3;\R)$ transgresses to $c_2(E)$, and Theorem~\ref{thm:contacttopology} gives $b_3(P)=b_3(T^4)=4$.  The base--fiber tensor product would have degree-three dimension $4+1=5$.  Thus even a contact parallel pair carries a nontrivial topological differential in its bundle cohomology.
\end{example}

\begin{example}[A maximal-rank constraint with an obstructed outer twist]\label{ex:SU3}
Write $T^6=T^2_1\times T^2_2\times T^2_3$, with integral area classes $a_1,a_2,a_3$.  Put $x=a_1+a_2$ and $y=a_3$.  Choose line bundles $L_x,L_y$ with these Chern classes and form
\[
 E=L_x\oplus L_y\oplus(L_x\otimes L_y)^{-1}.
\]
Its determinant is canonically trivial.  Let $P$ be the $\mathrm{SU}(3)$ frame bundle, choose diagonal unitary connections whose determinant connection is trivial, and let $u$ have diagonal eigenvalues $i,2i,-3i$.  The field is parallel and regular.  Its paired curvature, with a positive normalization of the area forms, is a nonzero multiple of
\[
 4a_1+4a_2+5a_3,
\]
which is symplectic.  The vertical block in~\eqref{eq:levi} has rank $\dim(\mathrm{SU}(3)/T^2)=6$; the horizontal block has rank six.  Thus $\mathcal D$ has Levi rank twelve and a one-dimensional characteristic space.

The sign fields $u$ and $-u$ lie in different adjoint orbits: the eigenvalue multisets $\{1,2,-3\}$ and $\{-1,-2,3\}$ differ.  Equivalently, the odd invariant $\operatorname{tr}(u^3)$ is nonzero.  Complex conjugation $\varphi(g)=\overline g$ is an outer automorphism of $\mathrm{SU}(3)$ and sends this diagonal $u$ to $-u$ under $\dd\varphi$.  It acts canonically on $P^\varphi$, where it also transforms $A$ and preserves the paired form by~\eqref{eq:twistnaturality}.  The twist cannot be realized on $P$ itself: the Whitney formula gives
\begin{equation}\label{eq:c3example}
 c_3(E)=xy(-x-y)=-2a_1a_2a_3\ne0,
\end{equation}
while the associated bundle of $P^\varphi$ is $\overline E$ and $c_3(\overline E)=-c_3(E)$.  A self-lift would force $2c_3(E)=0$, contradicting~\eqref{eq:c3example}.
\end{example}

The examples separate four properties: a sign change may exist as an operation on parallel fields, as a pointwise adjoint conjugacy, as a global gauge transformation, and as a symmetry of a fixed connection.  Automorphism twists add a fifth issue, the isomorphism class of the principal bundle.  The paired hyperplane sees the first operation, whereas the remaining properties require the orbit, reduction, holonomy, and characteristic data proved above.

\section{Type-A pointwise lift orders}

Adams and He~\cite{AdamsHe} analyze pointwise Weyl-lift orders for reductive groups. Theorem~\ref{thm:pushout} transfers those orders unchanged to every invariant torus reduction. The following elementary monomial-matrix calculation makes the result explicit for $\SU(n)$.

\begin{proposition}\label{prop:cycles}
Let $w\in W(\SU(n))=S_n$ have cycle lengths $d_1,\ldots,d_k$. Put
\[
 r=\operatorname{lcm}(d_1,\ldots,d_k).
\]
For every regular $\SU(n)$ field whose torus-bundle class is fixed by $w$, the least order of a gauge lift of $w$ is
\[
 \begin{cases}
 r,&\operatorname{sgn}(w)=+1\text{ or }r/d_j\text{ is even for some }j,\\
 2r,&\operatorname{sgn}(w)=-1\text{ and every }r/d_j\text{ is odd}.
 \end{cases}
\]
\end{proposition}

\begin{proof}
An element of the normalizer is a monomial unitary matrix. Let $\lambda_j\in S^1$ be the product of its nonzero entries along the $j$th cycle. The determinant-one condition is
$\prod_j\lambda_j=\operatorname{sgn}(w)$. Its $r$th power acts on the $j$th cycle by the scalar $\lambda_j^{r/d_j}$. If the permutation is even, choose every $\lambda_j=1$. If it is odd and some $r/d_j$ is even, choose $\lambda_j=-1$ on that cycle and $1$ on the others. These choices give order $r$. If the permutation is odd and all $r/d_j$ are odd, the product of elements of the finite groups $\mu_{r/d_j}$ cannot be $-1$, so an order-$r$ lift does not exist. Choosing one $\lambda_j=-1$ and the others $1$ always gives a lift of order dividing $2r$; since its projection has order $r$, its order is exactly $2r$ in the remaining case. Apply Theorem~\ref{thm:pushout}.
\end{proof}

\begin{example}
An $\SU(2)$ reflection lifts only to order four, while an $\SU(3)$
transposition lifts to order two. A four-cycle has least lift order
eight in $\SU(4)$ and order four after adding a fixed coordinate in
$\SU(5)$.
\end{example}

\enlargethispage{3\baselineskip}

\end{document}